\documentclass[a4paper,10pt]{amsart}

\usepackage{graphicx}

\usepackage[utf8]{inputenc}
\usepackage{verbatim}
\usepackage{amssymb,amsmath}
\usepackage{enumerate}
\usepackage[active]{srcltx}
\usepackage{mathabx}

\numberwithin{equation}{section}
\usepackage{float}
\usepackage[left=3.5 cm, right=3.5cm]{geometry}
\usepackage{MnSymbol}
 \usepackage[usenames,dvipsnames]{pstricks}
 \usepackage{epsfig}
\usepackage{pst-grad}
\usepackage{pstricks,pst-node}

\usepackage{xmpmulti}
\usepackage{psfrag}
\usepackage{tikz}

\usepackage{todonotes}

\title{Cycle reversions and dichromatic number in tournaments}
\date{\today}

\newtheorem{prop}{Proposition}[section]

\newtheorem{lemma}[prop]{Lemma}

\newtheorem{cor}[prop]{Corollary}
\newtheorem{con}[prop]{Conjecture}
\newtheorem{theorem}[prop]{Theorem}
\newtheorem{theorem?}[prop]{Theorem???}
\newtheorem{definition}[prop]{Definition}
\theoremstyle{definition}

\newtheorem*{mthm}{Main Theorem}
\newtheorem{claim}[prop]{Claim}
\newtheorem{obs}[prop]{Observation}

\newtheorem{prob}[prop]{Problem}

\DeclareMathOperator{\cf}{cf}

\newcommand{\lex}{<_{\scalebox{0.8}{\textrm{lex}}}}
\newcommand{\tri}{\vartriangleleft}
\newcommand{\la}{\langle}
\newcommand{\ra}{\rangle}
\DeclareMathOperator{\St}{St}
\newcommand{\mc}[1]{\mathcal{#1}}
\newcommand{\mb}[1]{\mathbb{#1}}

\newcommand{\oo}{\omega}
\newcommand{\uhr}{\upharpoonright}
\newcommand{\omg}{{\omega_1}}

\newcommand{\ran}{\text{ran}}

\newcommand{\mf}[1]{\mathfrak{#1}}
\newcommand{\crs}[1]{\mathfrak{r}(#1)}

\newcommand{\chr}[1]{\chi(#1)}
\newcommand{\dchr}[1]{\arr{\chi}(#1)}
\newcommand{\arr}[1]{\overrightarrow{#1}}
\newcommand{\larr}[1]{\overleftarrow{#1}}

\newcommand{\setm}{\setminus}

\newcommand{\rev}[2]{{#1}\hspace{-0.08 cm}\rcirclearrowright {#2}}
\newcommand{\smf}{\hspace{0.008 cm}^\smallfrown}

\newcommand{\rs}[1]{\mb{RS}({#1})}
\newcommand{\rsd}[1]{\mb{RS}^\Delta({#1})}

\author{Paul Ellis}
\address[P. Ellis]{Department of Mathematics and Computer Science, 
Manhattanville College, 2900 Purchase Street, Purchase, NY 10577, USA}
\email[]{paulellis@paulellis.org}
\urladdr{http://paulellis.org/}

\author{D\'aniel T. Soukup}
  \address[D.T. Soukup]{Universität Wien,
 Kurt Gödel Research Center for Mathematical Logic, Austria}
  \email[Corresponding author]{daniel.soukup@univie.ac.at}
\urladdr{http://www.logic.univie.ac.at/$\sim  $soukupd73/}

\newcommand{\DD}{\color{black}}

\subjclass[2010]{05C20, 05C63, 05C15, 05C38, 03E05}
\keywords{dichromatic number, acyclic, transitive, cycle reversion, tournaments}

\begin{document} 
  \begin{abstract}
  We show that if $D$ is a tournament of arbitrary size then $D$ has finite strong components after reversing a locally finite sequence of cycles. In turn, we prove that any tournament can be covered by two acyclic sets after reversing a locally finite sequence of cycles. This provides a partial solution to a conjecture of S. Thomass\'e.
 \end{abstract}
 
 \maketitle
 
\setcounter{tocdepth}{1} 
 \tableofcontents

 \section{Introduction}
 
 We are motivated by the following universal question: how regular can we make a complicated/random structure by a simple operation? A prime example in the setting of digraphs is the well-studied problem of finding \textit{feedback arc sets} (FAS, in short) \cite{festa}. A feedback arc set of a digraph $D$ {\DD (with arc set $A(D)$)} is a set $F\subseteq A(D)$ so that \textit{removing} $F$ from $A(D)$ makes the resulting digraph acyclic i.e. $A(D)\setm F$ contains no directed cycles.  Finding \emph{some} feedback arc set is easy, because one can linearly order the vertices and take all the backward pointing arcs. On the other hand, finding an FAS with the \emph{smallest possible size} is hard: this appears in Karp's famous list of NP-complete problems as number 8 \cite{karp} and the problem is still NP-hard even restricted for tournaments \cite{fasNP}.
 
 
 Now, it is easy to see that if $F$ is an FAS which is \textit{minimal} with respect to inclusion then it has the additional property that if one \textit{reverses} the arcs in $F$, instead of removing $F$,  then $D$ becomes acyclic i.e. $A(D)\setm F\cup \{vu:uv\in F\}$ is acyclic. The interested reader is referred to \cite{garth} and \cite{koh} for related discussions.
 
 Instead of reversing an arbitrary set of arcs, we consider  the following operation: given a digraph $D$, we take a directed cycle and \emph{reverse the orientation of arcs along this cycle}. That is, if $\mb C(D)$ denotes the directed cycles of a digraph $D$ and $C\in \mb C(D)$ then we let $$\rev{D}{C}$$ denote the digraph on vertices $V(D)$ and arcs $(A(D)\setm A(C))\cup \{vu:uv\in A(C)\}$. This is a rather benign operation compared to reversing an arbitrary set of arcs: the in-and out-degrees of any vertex in $D$ and $\rev{D}{C}$ are the same, as well as the  strong components of $D$ and $\rev{D}{C}$, and $\rev{D}{C}$ is definitely not acyclic.
 
Naturally, one can repeatedly turn cycles one after the other, always working in the resulting graph; we dubbed this operation \emph{sequential cycle reversion}. In the case of infinite graphs and infinitely many cycles, we require that each edge appears in only finitely many arcs in the sequence i.e. the sequence is \emph{locally finite}. This ensures that at each limit step we have a well-defined graph. How do we define limits of digraphs in this setting? Suppose that $\la D_\xi:\xi<\zeta\ra$ is a sequence of digraphs on the same vertex set $V$ and edge set $E$ so that for each $uv\in E$ there is a $\nu<\zeta$ so that either $uv\in A(D_\xi)$ or $vu\in A(D_\xi)$ for all $\nu<\xi<\zeta$ (i.e. each arc stabilizes eventually). Then let $D=\lim_{\xi<\zeta}D_\xi$ be the digraph with vertex set $V$ and edge set $E$ so that $uv\in A(D)$ iff  for some $\nu<\zeta$, $uv\in A(D_\xi)$ for all $\nu<\xi<\zeta$.
 
 To be completely precise, we can define $\rs D$, the \emph{reversal sequences} of $D$, and the reversed digraph $\rev{D}{\mc C}$ for $\mc C\in \rs D$ simultaneously as follows:
 
 \begin{enumerate}
  \item $\emptyset\in \rs D$ and $\rev{D}{\emptyset}=D$,
  \item if $\mc C\in \rs D$ and $C^*$ is a directed cycle in $\rev{D}{\mc C}$ then $\mc C^*=\mc C \smf \la C^*\ra\in \rs D$ and $\rev{D}{\mc C^*}=\rev{(\rev{D}{\mc C})}{C^*}$, and 
  \item if $\mc C_\xi\in \rs D$  for $\xi<\zeta$ is an increasing sequence so that $\mc C=\bigcup\{\mc C_\xi:\xi<\zeta\}$ is locally finite (i.e. each arc appears in only finitely many cycles in the sequence) then $\mc C\in \rs D$ and $\rev{D}{\mc C}=\lim_{\xi<\zeta}(\rev{D}{\mc C_\xi}).$ 
 \end{enumerate}

 Now we are interested in the following problem: given a complicated digraph $D$, how simple $\rev{D}{\mc C}$ can be for some $\mc C\in \rs D$? In particular, our main motivation is the following beautiful conjecture of S. Thomass\'e:
 
 \begin{con}\cite{thom}
  Given a (finite or infinite) digraph $D$, there is a $\mc C\in \rs D$ so that $\rev{D}{\mc C}$ is covered by two acyclic subgraphs.
 \end{con}
 
If we let $\dchr D $ denote the \emph{{\DD dichromatic} number} of $D$, that is, the least number of acyclic subgraphs needed to cover $D$ then the above question becomes the following: can we find $\mc C\in \rs D$ so that $\dchr{\rev{D}{\mc C}}$ is at most 2?  Note that a digraph is acyclic iff there is a linear order on the vertices so that arcs only point backward with respect to this order; these are in some sense structurally simple digraphs. In turn, the {\DD dichromatic} number measures how far we are from such a linear order. For the interested reader, finite digraphs with large {\DD dichromatic} number (and large digirth) can be constructed by various methods \cite{severino, bokal}; for the infinite counterparts, see \cite{dsoukup}. We also outline a very simple construction in Section \ref{sec:finite}.

The above conjecture for \emph{finite digraphs} was already resolved by Thomass\'e \footnote{Personal communication.} and a very elegant proof is presented by P. Charbit as well:
 
 \begin{theorem}\cite{charbit}\label{thm:charbit}
  For any finite digraph $D$, there is some $\mc C\in \rs D$ so that $\dchr{\rev{D}{\mc C}}\leq 2$.
 \end{theorem}
 
 {\DD Theorem \ref{thm:charbit} for finite \emph{tournaments} is discussed in detail in a paper of B. Guiduli, A. Gy\'arf\'as, S. Thomass\'e and P. Weidl \cite{gyarfas_rev}, with a focus on reversing 3-cycles\footnote{Acylic sets in the setting of tournaments are often called transitive.}. In part, their motivation comes from looking at tournaments with the same score sequence i.e. sequence of out-degrees; since each vertex has the same out-degree in $D$ and $\rev{D}{\mc C}$ (if $\mc C$ is finite), reversing cycles is a  way of realizing the same score sequence in different digraphs. Indeed, if $D$ and $D'$ have the same score sequence then there is a $\mc C\in \rs D$ so that $D'=\rev{D}{\mc C}$ (a result attributed to H. J. Ryser \cite{ryser} in \cite{gyarfas_rev}). See \cite{busch} for a more recent discussion on the topic of score sequences and dichromatic number in tournaments. 
 

 \medskip

  Now, our main goal will be to show the following rather unexpected result:
 
 \begin{mthm}
  If $D$ is a tournament of arbitrary size then there is $\mc C\in \rs D$ so that the strong components of $\rev{D}{\mc C}$ are finite. 
 \end{mthm}
 
  So no matter how random the initial tournament $D$ is, we can transform $D$ into a linear order \emph{modulo finite blocks} by cycle reversions.  Also, note that the strong components of  $\rev{D}{\mc C}$ must be contained in the strong components of $D$ and that a directed cycle must be contained in a strong component.  Hence the algorithm given by the the Main Theorem is not invertible: any further cycle reversion will only effect the finite blocks. {\DD This again is in stark contrast to the finite case where all cycle reversions are reversible, and strong components remain unchanged.
  
  Let us point out that along the proof of the Main Theorem, we uncover several auxilliary results about the structure of infinite tournaments which are of independent interest; we outline these results in Section \ref{sec:outline} as well.}

  Since Theorem \ref{thm:charbit} says that finite digraphs can be made to have {\DD dichromatic} number at most 2 via cycle reversions, we get the following corollary to our Main Theorem:
 
  \begin{cor}
  For any (finite or infinite) tournament $D$, there is some $\mc C\in \rs D$ so that $\dchr{\rev{D}{\mc C}}\leq 2$.
 \end{cor}

This corollary is again quite surprising: if $D$ is an infinite digraph of size $\kappa$ then $D$ might not be covered by $<\kappa$ many acyclic sets i.e. $\dchr{D}=\kappa$ can hold for a digraph or tournament of size $\kappa$ for any infinite $\kappa$ \cite{dsoukup}. Still, after reversing a locally finite sequence of cycles, the {\DD dichromatic} number can be lowered to 2.

\medskip

Our paper is structured as follows: first, in Section \ref{sec:outline}, we outline the plan of proving the Main Theorem and in Section \ref{sec:prelim}, we present some preliminary results which will be applied later multiple times. Next, Sections \ref{sec:parttree}-\ref{sec:unif} contain all the details of proving the Main Theorem. We only use rather basic set theoretic tools and technical complications only arise when dealing with digraphs of singular cardinality. In particular, the case of countably infinite tournaments is perfectly accessible for anyone without practice in infinite combinatorics.

We end our paper with three appendices: in Section \ref{sec:finite}, we prove a strengthening of Theorem \ref{thm:charbit} for finite tournaments. {\DD Furthermore, we will reflect on Charbit's argument and review \cite{gyarfas_rev} to see how many cycles are required to lower the chromatic number and how fast an algorithm can be to carry this out. }
In Section \ref{app:structure}, we show that any reversal sequence (no matter how large) is the composition of edge-disjoint \emph{countable} reversal sequences,  and finally, in Section \ref{app:tri}, we see if an arbitrary cycle reversion is equivalent to reversing a sequence of 3-cycles only.

 \medskip
 Finally, we would like to emphasize that Thomass\'e's conjecture remains unresolved in general, even for countably infinite digraphs. We included further questions and remarks on finite and infinite digraphs at the end of our paper.

 \subsection*{Notations and terminology} In this paper, when we write \emph{digraph}, we always mean an oriented simple graph, that is, one with no loops or parallel arcs.  If $u,v$ are vertices, then we use the word \emph{arc} to refer to the directed edge $uv$, and the word \emph{edge} to refer to the unordered edge $\{u,v\}$.  Let $E(D)$ denote the set of edges of $D$, $A(D)$ denote the set of arcs of $D$, and $d^+(v)$, $d^-(v)$ denote the outdegree and indegree of $v$, respectively. Our paper mostly concerns tournaments i.e. exactly one of $uv$, $vu$ is an arc of $D$ whenever $u\neq v\in V(D)$. 
 
 {\DD If $A,B\subset V(D)$ then let $\arr{AB}$ denote $\{uv\in A(D):u\in A,v\in B\}$. We use $D\uhr A$ to denote the induced subdigraph on $A$. A digraph $D$ is \emph{strongly connected} or just \emph{strong} if any two distinct points are connected by a directed path. The strong components of $D$ are the maximal strong induced subdigraphs of $D$; the strong components form a partition of $D$.
 
 Throughout the paper we will use standard notations and facts from set theory which can be found in \cite{jech}. In particular, we identify a number (or ordinal) $n$ with the set of all smaller ordinals $\{0,1\dots n-1\}$.
 }

 {\DD
 \subsection*{Acknowledgements} We would like to thank C. Laflamme, A. A. Lopez, N. Sauer and R. Woodrow for communicating Thomass\'e's conjecture and for inspiring discussions while the second author was a PIMS postdoctoral fellow at the University of Calgary.
	
	\smallskip
	
	The second named author was supported in part by PIMS and the FWF Grant I1921. }

%
%
%
%
%
 
\section{An outline of the main theorem}\label{sec:outline}

 Our proof of the Main Theorem breaks down into three major parts which we outline first before presenting the details. 
 
 We will proceed by induction on the size of the infinite tournament $D$, denoted by $\kappa$ in what follows. So, it suffices to show that strong components of $D$ can be made of size $<\kappa$ by cycle reversion; then the induction applies to these components and we have finite strong components in the end.

       \subsection*{Part I - outline} Let $\mf T_\kappa$ denote the class of tournaments of size $\kappa$. {\DD The first structural result we prove is the following:} if $D\in \mf T_\kappa$ then $V=V(D)$ can be partitioned into pieces $\{V_x:x\in \Gamma\}$ where $(\Gamma,\tri)$ is a linear order so that 
     \begin{enumerate}
      \item $D\uhr V_x$ is \emph{ $\kappa$-uniform} i.e. 
      \[ |(N^{+}(v)\Delta N^{+}(v'))\cap V_x|<\kappa \textmd{  and }|(N^{-}(v)\Delta N^{-}(v'))\cap V_x|<\kappa\] 
      for all $v,v'\in V_x$\footnote{Recall that $A\Delta B$ denotes $A\setm B\sup B\setm A.$}, and 
      \item each arc $uv\in \arr{V_y V_x}$ such that $x\tri y$ is \emph{$\kappa$-reversible} i.e. there are $\kappa$ many paths from $v$ to $u$ which are pairwise edge disjoint.
     \end{enumerate}

     In turn, we will deduce that if $D\in \mf T_\kappa$ then there is $\mc C\in \rs D$ so that strong components of $\rev{D}{\mc C}$ are $\kappa$-uniform.
     
     \medskip
     
     Let $\mf{UT}_\kappa$ denote the class of $\kappa$-uniform tournaments of size $\kappa$.  Part I implies that we can focus solely on $D\in \mf{UT}_\kappa$ instead of arbitrary tournaments.  

            \subsection*{Part II - outline} We deal with a special type of $\kappa$-uniform tournaments first: when all in or all out-degrees are small i.e. less than $\kappa$.     
            
            We show that if all out-degrees are $<\kappa$ in $D$ then after an appropriate sequential cycle reversion $D$ has a well ordered block structure with arcs between the blocks pointing backward. More precisely, there is $\mc C\in \rs D$ so we can write $V(D)$ as $\bigsqcup\{V_\xi:\xi<\mu\}$ where  $|V_\xi|<\kappa$ and  in $\rev{D}{\mc C}$, all arcs between the blocks $V_\zeta$ and $V_\xi$ for $\zeta<\xi$ point into $V_\zeta$. Figure \ref{fig:smalldeg} shows how $\rev{D}{\mc C}$ resembles the simplest tournament with out-degrees $<\kappa$ where the vertices are ordered in type $\kappa$ and all arcs point backward.
            
            \begin{figure}[H]

            \psscalebox{0.7} 
{
\begin{pspicture}(0,-1.7045493)(18.003002,1.7045493)
\rput[bl](14.24,-0.88454926){$V_\xi$}
\rput[bl](10.4,-0.8045493){$V_\zeta$}
\rput[bl](17.4,1.1754508){$V$}
\rput{45.51191}(3.7124248,-11.999647){\psarc[linecolor=black, linewidth=0.04, dimen=outer, arrowsize=0.05291667cm 2.0,arrowlength=1.4,arrowinset=0.0]{->}(16.16,-1.5745492){2.61}{0.0}{90.0}}
\rput{43.26387}(1.9352919,-7.029139){\psarc[linecolor=black, linewidth=0.04, dimen=outer, arrowsize=0.05291667cm 2.0,arrowlength=1.4,arrowinset=0.0]{->}(9.83,-1.0745492){1.84}{0.0}{90.0}}
\rput{46.065075}(2.2731495,-10.09586){\psarc[linecolor=black, linewidth=0.04, dimen=outer, arrowsize=0.05291667cm 2.0,arrowlength=1.4,arrowinset=0.0]{->}(13.01,-2.3745492){3.7}{0.0}{90.0}}
\psellipse[linecolor=black, linewidth=0.04, dimen=outer](14.46,0.25545076)(1.9,0.6)
\psellipse[linecolor=black, linewidth=0.04, dimen=outer](10.66,0.25545076)(0.9,0.4)
\psline[linecolor=black, linewidth=0.04, linestyle=dashed, dash=0.17638889cm 0.10583334cm](0.0,0.035450745)(7.2,0.035450745)
\psdots[linecolor=black, dotsize=0.16](1.8,0.035450745)
\psdots[linecolor=black, dotsize=0.16](4.8,0.035450745)
\rput[bl](1.8,-0.56454927){$\zeta$}
\rput[bl](4.8,-0.56454927){$\xi$}
\rput[bl](7.0,0.43545073){$\kappa$}
\rput{43.26387}(-0.07526917,-2.6392994){\psarc[linecolor=black, linewidth=0.04, dimen=outer, arrowsize=0.05291667cm 2.0,arrowlength=1.4,arrowinset=0.0]{->}(3.29,-1.4145492){2.08}{0.0}{90.0}}
\rput{43.546757}(-0.7936283,-2.2822094){\psarc[linecolor=black, linewidth=0.04, dimen=outer, arrowsize=0.05291667cm 2.0,arrowlength=1.4,arrowinset=0.0]{->}(2.46,-2.1345491){3.19}{0.0}{90.0}}
\rput{44.3405}(-0.33817512,-3.4791942){\psarc[linecolor=black, linewidth=0.04, dimen=outer, arrowsize=0.05291667cm 2.0,arrowlength=1.4,arrowinset=0.0]{->}(4.1,-2.1545494){3.27}{0.0}{90.0}}
\rput{-133.69044}(23.013718,11.893101){\psarc[linecolor=black, linewidth=0.04, dimen=outer, arrowsize=0.05291667cm 2.0,arrowlength=1.4,arrowinset=0.0]{->}(14.05,1.0254507){1.12}{0.0}{90.0}}
\end{pspicture}
}
 \caption{Tournaments with small outdegree}
\label{fig:smalldeg}
                     \end{figure}
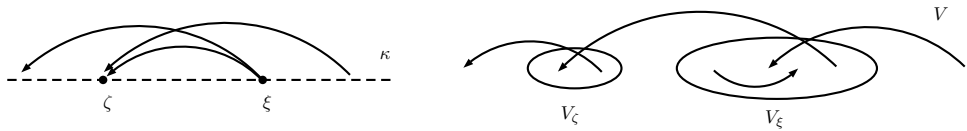
           
            Similarly, if all in-degrees are $<\kappa$ in $D$ then we can write $V(D)$ as $\bigsqcup\{V_\xi:\xi<\mu\}$ so that $|V_\xi|<\kappa$ and in $\rev{D}{\mc C}$, all arcs between blocks point forward.

            
            We remark that if $\kappa$ is uncountable and regular then we actually don't need to reverse any cycles to construct our blocks. The countable and the singular uncountable cases, however, are more tricky and require cycle reversions. This is our second structural result.
     
                 Ultimately, the strong components of $\rev{D}{\mc C}$ must be contained in these $V_\xi$ blocks in both cases.


 \subsection*{Part III - outline}  Now, suppose that $D\in \mf{UT}_\kappa$ is arbitrary.  We show how to make the strong components have size $<\kappa$ building on Part II. 
 
  We first show that $V(D)$ can be partitioned into $V^+\cup V^-$ so that in-degrees are $<\kappa$ in $D\uhr V^+$, out-degrees are $<\kappa$ in $D\uhr V^-$, and either (a) each arc in $\arr{V^-V^+}$ is $\kappa$-reversible, or (b) there is a vertex set $W\subseteq V$ of size $<\kappa$ which meets all arcs in $\arr{V^+V^-}$.
 
 If (a) holds, then using Lemma \ref{lm:kapparev}, we can turn all the arcs in $\arr{V^-V^+}$ without changing anything else and so strong components will be either inside $V^+$ or $V^-$. In $V^+$ or $V^-$, all in or all out-degrees are $<\kappa$ so we can make the strong components have size $<\kappa$ by Part II.

 If (b) holds, we will turn some cycles and modify $V^+$ and $V^-$ a bit further to cover $V$, so that all arcs across the new partition will point in the same direction, while the modified $V^+$ and $V^-$ satisfy the conditions of Part II. Hence, strong components can be made small again. Difficulties only arise really if $\kappa$ is singular.
 
 \section{Further preliminaries}\label{sec:prelim}
 
 Let us present a few simple results which will play a crucial role in proving the Main Theorem. First, note that if $D$ is a digraph and $\mc C\in \rs D$ is \emph{finite} then 
 \begin{enumerate}
  \item if $D$ was not acyclic then $\rev{D}{\mc C}$ is not acyclic either, and
  \item $d^+_D(v)=d^+_{\rev{D}{\mc C}}(v)$ and  $d^-_D(v)=d^-_{\rev{D}{\mc C}}(v)$.
 \end{enumerate}

 The infinite case is completely different and, in particular, both of the above observations fail in general.
 
 \begin{obs}\label{widget1}
  Suppose that $uv\in A$ and there are infinitely many edge-disjoint paths from $v$ to $u$ in $D$. Then there is $\mc C\in \rs D$ so that $A(\rev{D}{\mc C})=(A(D)\setm \{uv\})\cup \{vu\}$.
 \end{obs}
In other words, we can reverse $uv$ by cycle reversions without changing any other arc. We use the following notation: if $P$ is a sequence of vertices $v_0 v_1\dots v_k$ (e.g. a path) then let $\larr P$ denote $v_k v_{k-1}\dots v_0$. So if $C\in \mb C(D)$, $u\neq v\in V(C)$ and $P=uCv$ {\DD (i.e. $P$ is the unique subpath of $C$ from $u$ to $v$)} then $\larr P$ is a path in $\rev{D}{C}$.

\begin{proof}
 Let $\{P_n:n\in \oo\}$ be edge-disjoint paths from $v$ to $u$ and define $C_n$ as follows: $C_0=uvP_0$, $C_1=\larr P_0 P_1$, $C_2=\larr P_1 P_2$\dots and $C_{n+1}=\larr P_n P_{n+1}$ in general (see Figure \ref{fig:widget1}). It is clear that $\mc C=\la C_n\ra_{n\in \oo}$ is as desired.

 \begin{figure}[H]
  \psscalebox{0.7} 
{
\begin{pspicture}(0,-5.6367674)(20.24,5.6367674)
\psdots[linecolor=black, dotsize=0.1](18.82,2.3632324)
\psdots[linecolor=black, dotsize=0.1](18.42,2.3632324)
\psdots[linecolor=black, dotsize=0.1](18.02,2.3632324)
\psbezier[linecolor=black, linewidth=0.04, arrowsize=0.05291667cm 2.0,arrowlength=1.4,arrowinset=0.0]{->}(10.78,0.8232324)(11.78,0.8232324)(16.83925,1.4385543)(17.22,2.363232421875)(17.60075,3.2879105)(12.36,4.6832323)(11.22,4.3232327)
\psbezier[linecolor=black, linewidth=0.04, linestyle=dashed, dash=0.17638889cm 0.10583334cm, arrowsize=0.05291667cm 2.0,arrowlength=1.4,arrowinset=0.0]{->}(10.76,0.9432324)(11.76,1.1432325)(14.823884,1.3826517)(15.02,2.363232421875)(15.216116,3.3438132)(11.74,4.263232)(10.74,3.8632324)
\psbezier[linecolor=black, linewidth=0.04, linestyle=dashed, dash=0.17638889cm 0.10583334cm, arrowsize=0.05291667cm 2.0,arrowlength=1.4,arrowinset=0.38]{<-}(10.7,1.0832324)(11.7,1.2832325)(12.22,1.5632324)(12.42,2.163232421875)(12.62,2.7632325)(11.74,3.5432324)(10.74,3.5432324)
\psdots[linecolor=black, dotsize=0.2](10.42,3.5632324)
\psdots[linecolor=black, dotsize=0.2](10.42,0.9632324)
\psline[linecolor=black, linewidth=0.04, arrowsize=0.05291667cm 2.0,arrowlength=1.4,arrowinset=0.0]{->}(10.42,0.9632324)(10.42,3.3632324)
\psdots[linecolor=black, dotsize=0.2](0.82,0.9632324)
\psline[linecolor=black, linewidth=0.04, linestyle=dashed, dash=0.17638889cm 0.10583334cm, arrowsize=0.05291667cm 2.0,arrowlength=1.4,arrowinset=0.0]{<-}(0.82,1.1632324)(0.82,3.3632324)
\psdots[linecolor=black, dotsize=0.2](0.82,3.5632324)
\psbezier[linecolor=black, linewidth=0.04, linestyle=dashed, dash=0.17638889cm 0.10583334cm, arrowsize=0.05291667cm 2.0,arrowlength=1.4,arrowinset=0.38]{->}(1.22,1.1632324)(2.02,1.3477285)(2.66,1.7167208)(2.82,2.2702091660610466)(2.98,2.8236976)(2.276,3.5432324)(1.476,3.5432324)
\psbezier[linecolor=black, linewidth=0.04, arrowsize=0.05291667cm 2.0,arrowlength=1.4,arrowinset=0.0]{->}(1.24,1.0432324)(2.24,1.2432324)(5.2238836,1.3826517)(5.42,2.363232421875)(5.616116,3.3438132)(2.14,4.263232)(1.14,3.8632324)
\psbezier[linecolor=black, linewidth=0.04, arrowsize=0.05291667cm 2.0,arrowlength=1.4,arrowinset=0.0]{->}(1.2,0.8832324)(2.2,0.8832324)(7.23925,1.4385543)(7.62,2.363232421875)(8.00075,3.2879105)(2.76,4.6832323)(1.62,4.3232327)
\rput[bl](3.64,5.2432323){\Large{$D$}}
\psdots[linecolor=black, dotsize=0.1](8.42,2.3632324)
\psdots[linecolor=black, dotsize=0.1](8.82,2.3632324)
\psdots[linecolor=black, dotsize=0.1](9.22,2.3632324)
\rput[bl](13.42,5.3632326){\Large{$\rev{D}{C_0}$}}
\rput[bl](1.82,-1.0367676){\Large{$\rev{\rev{D}{C_0}}{C_1}$}}
\psdots[linecolor=black, dotsize=0.1](9.22,-4.0367675)
\psdots[linecolor=black, dotsize=0.1](8.82,-4.0367675)
\psdots[linecolor=black, dotsize=0.1](8.42,-4.0367675)
\psbezier[linecolor=black, linewidth=0.04, linestyle=dashed, dash=0.17638889cm 0.10583334cm, arrowsize=0.05291667cm 2.0,arrowlength=1.4,arrowinset=0.0]{->}(1.14,-5.5367675)(2.14,-5.5367675)(7.23925,-4.961446)(7.62,-4.036767578125)(8.00075,-3.1120894)(2.76,-1.7167675)(1.62,-2.0767677)
\psbezier[linecolor=black, linewidth=0.04, linestyle=dashed, dash=0.17638889cm 0.10583334cm, arrowsize=0.05291667cm 2.0,arrowlength=1.4,arrowinset=0.0]{<-}(1.1,-5.4167676)(2.1,-5.216768)(5.2238836,-5.0173483)(5.42,-4.036767578125)(5.616116,-3.056187)(2.14,-2.1367676)(1.14,-2.5367675)
\psbezier[linecolor=black, linewidth=0.04, arrowsize=0.05291667cm 2.0,arrowlength=1.4,arrowinset=0.38]{->}(1.08,-5.2767677)(2.08,-5.0767674)(2.62,-4.8367677)(2.82,-4.236767578125)(3.02,-3.6367676)(2.14,-2.8567677)(1.14,-2.8567677)
\psdots[linecolor=black, dotsize=0.2](0.82,-2.8367677)
\psline[linecolor=black, linewidth=0.04, arrowsize=0.05291667cm 2.0,arrowlength=1.4,arrowinset=0.0]{->}(0.82,-5.4367676)(0.82,-3.0367675)
\psdots[linecolor=black, dotsize=0.2](0.82,-5.4367676)
\psdots[linecolor=black, dotsize=0.1](19.02,-4.0367675)
\psdots[linecolor=black, dotsize=0.1](18.62,-4.0367675)
\psdots[linecolor=black, dotsize=0.1](18.22,-4.0367675)
\psbezier[linecolor=black, linewidth=0.04, linestyle=dashed, dash=0.17638889cm 0.10583334cm, arrowsize=0.05291667cm 2.0,arrowlength=1.4,arrowinset=0.0]{<-}(10.9,-5.6167674)(11.9,-5.6167674)(17.03925,-4.961446)(17.42,-4.036767578125)(17.80075,-3.1120894)(12.56,-1.7167675)(11.42,-2.0767677)
\psbezier[linecolor=black, linewidth=0.04, arrowsize=0.05291667cm 2.0,arrowlength=1.4,arrowinset=0.0]{->}(10.98,-5.4967675)(11.98,-5.2967677)(15.023884,-5.0173483)(15.22,-4.036767578125)(15.416116,-3.056187)(11.94,-2.1367676)(10.94,-2.5367675)
\psbezier[linecolor=black, linewidth=0.04, arrowsize=0.05291667cm 2.0,arrowlength=1.4,arrowinset=0.38]{->}(10.96,-5.3767676)(11.96,-5.1767673)(12.42,-4.8367677)(12.62,-4.236767578125)(12.82,-3.6367676)(11.94,-2.8567677)(10.94,-2.8567677)
\psdots[linecolor=black, dotsize=0.2](10.62,-2.8367677)
\psline[linecolor=black, linewidth=0.04, arrowsize=0.05291667cm 2.0,arrowlength=1.4,arrowinset=0.0]{->}(10.62,-5.4367676)(10.62,-3.0367675)
\psdots[linecolor=black, dotsize=0.2](10.62,-5.4367676)
\rput[bl](11.62,-1.0367676){\Large{$\rev{\rev{\rev{D}{C_0}}{C_1}}{C_2}$}}
\rput[bl](0.02,3.6432323){\Large{$u$}}
\rput[bl](0.02,1.0032325){\Large{$v$}}
\rput[bl](0.0,-2.7967675){\Large{$u$}}
\rput[bl](0.0,-5.4367676){\Large{$v$}}
\end{pspicture}
}\caption{The dashed paths mark the next cycle to be reversed.}
\label{fig:widget1}
 \end{figure}
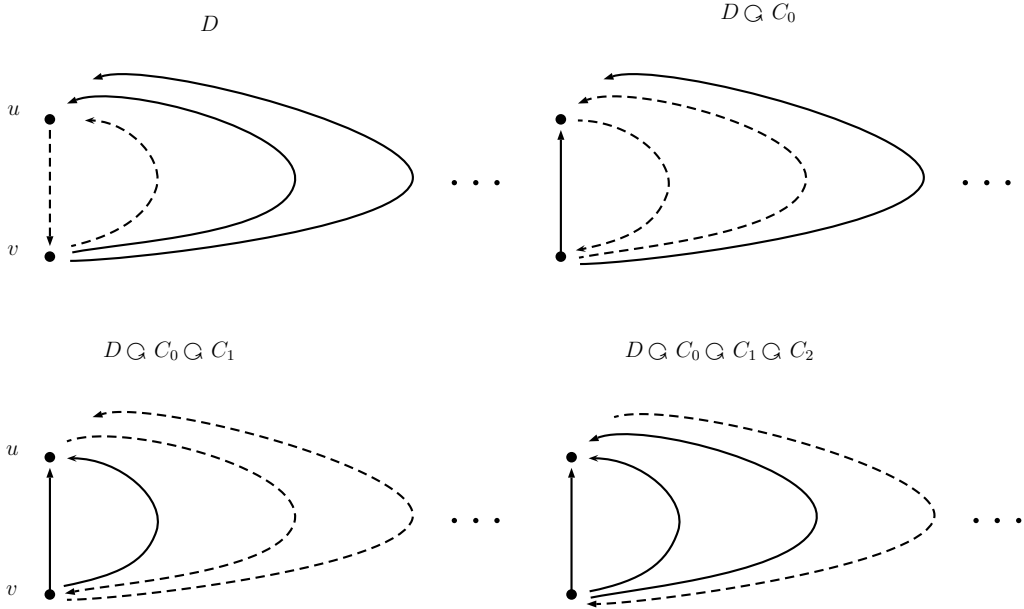

\end{proof}

The above observation motivates the next definition.

\begin{definition}
 An arc $uv$ in a digraph $D$ is \emph{$\kappa$-reversible} iff there are $\kappa$ many paths from $v$ to $u$ which are pairwise edge disjoint.
\end{definition}

It is easy to see that one can simultaneously reverse $\kappa$ many $\kappa$-reversible arcs:

\begin{lemma}\label{lm:kapparev}
 Suppose that $D$ is a digraph and $F$ is a set of at most $\kappa$ many $\kappa$-reversible arcs (where $\kappa$ is an infinite cardinal). Then there is $\mc C\in \rs D$ so that  $A(\rev{D}{\mc C})=(A(D)\setm F)\cup \{vu:uv\in F\}$.
\end{lemma}


\begin{proof}
List $F$ as $\{u_\xi v_\xi:\xi<\kappa\}$.  For each $\xi<\kappa$ and each $n<\omega$, choose a path $P_\xi^n$ from $v_\xi$ to $u_\xi$ in such a way so that the set of all these paths is pairwise edge disjoint.  Then let $\mc C_{\xi}$ be the reversal sequence given by Observation \ref{widget1} with respect to the arc $u_\xi v_\xi$ and the paths $\{ P_\xi^n \mid n<\omega \}$.  While $u_\xi v_\xi$ may occur in some $P_{\xi'}^n$, it can only appear in one such path, so $\la C_{\xi}\ra_{\xi < \kappa}$ is locally finite, and thus reverses exactly the arcs in $F$. 
\end{proof}

 \begin{cor}
  If $D$ is a digraph of size $\kappa$ and each arc of $D$ is $\kappa$-reversible then for any other digraph $D^*$ on the same vertex and edge set there is a $\mc C\in \rs D$ so that $D^*=\rev{D}{\mc C}$.
 \end{cor}

 In particular, we can make $D$ acyclic by cycle reversions if each arc is $\kappa$-reversible. This is the case for a good number of natural examples e.g. the random digraph for $\kappa=\aleph_0$, and for various uncountable digraphs with large {\DD dichromatic} number which are defined by forcing or using strong colourings \cite{dsoukup}.
 
  \medskip
 
 Now, we proceed by a similar example to Observation \ref{widget1} but we will reverse the arcs along an infinite path instead of a single arc:
 
 \begin{obs}\label{widget2}
  Suppose that $v_0 v_1 \dots$ is  a one-way infinite directed path in $D$ i.e. $v_i v_{i+1}\in A(D)$ for all $i<\oo$, and $N^-(v_0)\cap \{v_n:n<\oo\}$ is infinite. Then there is a $\mc C\in \rs D$ so that $A(\rev{D}{\mc C})=(A(D)\setm \{v_i v_{i+1}:i<\oo\})\cup \{v_{i+1} v_{i}:i<\oo\}$.
 \end{obs}
 
 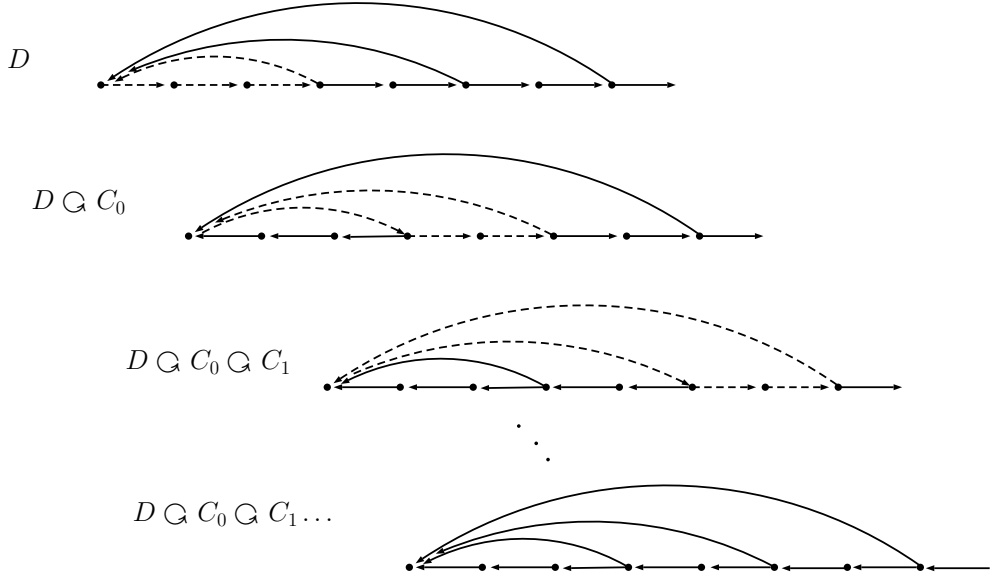
\begin{figure}[H]
  \psscalebox{0.6} 
{
\begin{pspicture}(0,-6.5789785)(21.54,6.5789785)
\rput{22.33402}(-0.82607305,-2.555646){\psarc[linecolor=black, linewidth=0.04, dimen=outer, arrowsize=0.05291667cm 2.0,arrowlength=1.4,arrowinset=0.0]{->}(6.06,-3.370132){8.53}{39.982376}{91.207924}}
\rput{54.82714}(0.27652952,-7.8334346){\psarc[linecolor=black, linewidth=0.04, dimen=outer, arrowsize=0.05291667cm 2.0,arrowlength=1.4,arrowinset=0.0]{->}(7.69,-3.650132){9.62}{0.0}{70.0}}
\rput{29.463062}(0.69166416,-2.210832){\psarc[linecolor=black, linewidth=0.04, linestyle=dashed, dash=0.17638889cm 0.10583334cm, dimen=outer, arrowsize=0.05291667cm 2.0,arrowlength=1.4,arrowinset=0.0]{->}(4.55,0.20986801){4.58}{30.977716}{89.12375}}
\psline[linecolor=black, linewidth=0.04, arrowsize=0.05291667cm 2.0,arrowlength=1.4,arrowinset=0.0]{->}(13.26,4.1598682)(14.66,4.1598682)
\psdots[linecolor=black, dotsize=0.16](13.26,4.1598682)
\psline[linecolor=black, linewidth=0.04, arrowsize=0.05291667cm 2.0,arrowlength=1.4,arrowinset=0.0]{->}(11.66,4.1598682)(13.06,4.1598682)
\psdots[linecolor=black, dotsize=0.16](11.66,4.1598682)
\psline[linecolor=black, linewidth=0.04, arrowsize=0.05291667cm 2.0,arrowlength=1.4,arrowinset=0.0]{->}(10.06,4.1598682)(11.46,4.1598682)
\psdots[linecolor=black, dotsize=0.16](10.06,4.1598682)
\psline[linecolor=black, linewidth=0.04, arrowsize=0.05291667cm 2.0,arrowlength=1.4,arrowinset=0.0]{->}(8.46,4.1598682)(9.86,4.1598682)
\psdots[linecolor=black, dotsize=0.16](8.46,4.1598682)
\psline[linecolor=black, linewidth=0.04, arrowsize=0.05291667cm 2.0,arrowlength=1.4,arrowinset=0.0]{->}(6.86,4.1598682)(8.26,4.1598682)
\psdots[linecolor=black, dotsize=0.16](6.86,4.1598682)
\psline[linecolor=black, linewidth=0.04, linestyle=dashed, dash=0.17638889cm 0.10583334cm, arrowsize=0.05291667cm 2.0,arrowlength=1.4,arrowinset=0.0]{->}(5.26,4.1598682)(6.66,4.1598682)
\psdots[linecolor=black, dotsize=0.16](5.26,4.1598682)
\psline[linecolor=black, linewidth=0.04, linestyle=dashed, dash=0.17638889cm 0.10583334cm, arrowsize=0.05291667cm 2.0,arrowlength=1.4,arrowinset=0.0]{->}(3.66,4.1598682)(5.06,4.1598682)
\psdots[linecolor=black, dotsize=0.16](3.66,4.1598682)
\psline[linecolor=black, linewidth=0.04, linestyle=dashed, dash=0.17638889cm 0.10583334cm, arrowsize=0.05291667cm 2.0,arrowlength=1.4,arrowinset=0.0]{->}(2.06,4.1598682)(3.46,4.1598682)
\psdots[linecolor=black, dotsize=0.16](2.06,4.1598682)
\rput{22.33402}(-1.9520112,-3.5320091){\psarc[linecolor=black, linewidth=0.04, linestyle=dashed, dash=0.17638889cm 0.10583334cm, dimen=outer, arrowsize=0.05291667cm 2.0,arrowlength=1.4,arrowinset=0.0]{->}(7.97,-6.710132){8.54}{39.982376}{91.207924}}
\rput{54.82714}(-1.6396528,-10.818886){\psarc[linecolor=black, linewidth=0.04, dimen=outer, arrowsize=0.05291667cm 2.0,arrowlength=1.4,arrowinset=0.0]{->}(9.61,-6.990132){9.62}{0.0}{70.0}}
\rput{29.463062}(-0.7028482,-3.58716){\psarc[linecolor=black, linewidth=0.04, linestyle=dashed, dash=0.17638889cm 0.10583334cm, dimen=outer, arrowsize=0.05291667cm 2.0,arrowlength=1.4,arrowinset=0.0]{<-}(6.47,-3.130132){4.58}{30.977716}{89.12375}}
\psline[linecolor=black, linewidth=0.04, arrowsize=0.05291667cm 2.0,arrowlength=1.4,arrowinset=0.0]{->}(15.18,0.819868)(16.58,0.819868)
\psdots[linecolor=black, dotsize=0.16](15.18,0.819868)
\psline[linecolor=black, linewidth=0.04, arrowsize=0.05291667cm 2.0,arrowlength=1.4,arrowinset=0.0]{->}(13.58,0.819868)(14.98,0.819868)
\psdots[linecolor=black, dotsize=0.16](13.58,0.819868)
\psline[linecolor=black, linewidth=0.04, arrowsize=0.05291667cm 2.0,arrowlength=1.4,arrowinset=0.0]{->}(11.98,0.819868)(13.38,0.819868)
\psdots[linecolor=black, dotsize=0.16](11.98,0.819868)
\psline[linecolor=black, linewidth=0.04, linestyle=dashed, dash=0.17638889cm 0.10583334cm, arrowsize=0.05291667cm 2.0,arrowlength=1.4,arrowinset=0.0]{->}(10.38,0.819868)(11.78,0.819868)
\psdots[linecolor=black, dotsize=0.16](10.38,0.819868)
\psline[linecolor=black, linewidth=0.04, linestyle=dashed, dash=0.17638889cm 0.10583334cm, arrowsize=0.05291667cm 2.0,arrowlength=1.4,arrowinset=0.0]{->}(8.78,0.819868)(10.18,0.819868)
\psdots[linecolor=black, dotsize=0.16](8.78,0.819868)
\psline[linecolor=black, linewidth=0.04, arrowsize=0.05291667cm 2.0,arrowlength=1.4,arrowinset=0.0]{<-}(7.34,0.799868)(8.72,0.819868)
\psdots[linecolor=black, dotsize=0.16](7.18,0.819868)
\psline[linecolor=black, linewidth=0.04, arrowsize=0.05291667cm 2.0,arrowlength=1.4,arrowinset=0.0]{<-}(5.76,0.819868)(7.16,0.819868)
\psdots[linecolor=black, dotsize=0.16](5.58,0.819868)
\psline[linecolor=black, linewidth=0.04, arrowsize=0.05291667cm 2.0,arrowlength=1.4,arrowinset=0.0]{<-}(4.12,0.819868)(5.52,0.819868)
\psdots[linecolor=black, dotsize=0.16](3.98,0.819868)
\psline[linecolor=black, linewidth=0.04, arrowsize=0.05291667cm 2.0,arrowlength=1.4,arrowinset=0.0]{->}(18.22,-2.520132)(19.62,-2.520132)
\rput[bl](0.0,4.539868){\huge{$D$}}
\rput[bl](0.54,1.259868){\huge{$\rev{D}{C_0}$}}
\psdots[linecolor=black, dotsize=0.08](11.22,-3.380132)
\rput[bl](2.6,-2.280132){\huge{$\rev{D}{C_0}\rev{}{C_1}$}}
\rput{22.33402}(-2.9931815,-4.9377785){\psarc[linecolor=black, linewidth=0.04, linestyle=dashed, dash=0.17638889cm 0.10583334cm, dimen=outer, arrowsize=0.05291667cm 2.0,arrowlength=1.4,arrowinset=0.0]{<-}(11.01,-10.050132){8.54}{39.982376}{91.207924}}
\rput{54.82714}(-3.0810058,-14.719845){\psarc[linecolor=black, linewidth=0.04, linestyle=dashed, dash=0.17638889cm 0.10583334cm, dimen=outer, arrowsize=0.05291667cm 2.0,arrowlength=1.4,arrowinset=0.0]{->}(12.65,-10.330132){9.62}{0.0}{70.0}}
\rput{29.463062}(-1.9525143,-5.5143743){\psarc[linecolor=black, linewidth=0.04, dimen=outer, arrowsize=0.05291667cm 2.0,arrowlength=1.4,arrowinset=0.0]{->}(9.51,-6.470132){4.58}{30.977716}{89.12375}}
\psdots[linecolor=black, dotsize=0.16](18.22,-2.520132)
\psline[linecolor=black, linewidth=0.04, linestyle=dashed, dash=0.17638889cm 0.10583334cm, arrowsize=0.05291667cm 2.0,arrowlength=1.4,arrowinset=0.0]{->}(16.62,-2.520132)(18.02,-2.520132)
\psdots[linecolor=black, dotsize=0.16](16.62,-2.520132)
\psline[linecolor=black, linewidth=0.04, linestyle=dashed, dash=0.17638889cm 0.10583334cm, arrowsize=0.05291667cm 2.0,arrowlength=1.4,arrowinset=0.0]{->}(15.02,-2.520132)(16.42,-2.520132)
\psdots[linecolor=black, dotsize=0.16](15.02,-2.520132)
\psline[linecolor=black, linewidth=0.04, arrowsize=0.05291667cm 2.0,arrowlength=1.4,arrowinset=0.0]{<-}(13.62,-2.520132)(14.94,-2.520132)
\psdots[linecolor=black, dotsize=0.16](13.42,-2.520132)
\psline[linecolor=black, linewidth=0.04, arrowsize=0.05291667cm 2.0,arrowlength=1.4,arrowinset=0.0]{<-}(11.96,-2.520132)(13.38,-2.520132)
\psdots[linecolor=black, dotsize=0.16](11.82,-2.520132)
\psline[linecolor=black, linewidth=0.04, arrowsize=0.05291667cm 2.0,arrowlength=1.4,arrowinset=0.0]{<-}(10.38,-2.540132)(11.76,-2.520132)
\psdots[linecolor=black, dotsize=0.16](10.22,-2.520132)
\psline[linecolor=black, linewidth=0.04, arrowsize=0.05291667cm 2.0,arrowlength=1.4,arrowinset=0.0]{<-}(8.8,-2.520132)(10.2,-2.520132)
\psdots[linecolor=black, dotsize=0.16](8.62,-2.520132)
\psline[linecolor=black, linewidth=0.04, arrowsize=0.05291667cm 2.0,arrowlength=1.4,arrowinset=0.0]{<-}(7.16,-2.520132)(8.56,-2.520132)
\psdots[linecolor=black, dotsize=0.16](7.02,-2.520132)
\rput[bl](2.76,-5.640132){\huge{$\rev{D}{C_0}\rev{}{C_1}\dots$}}
\rput{22.33402}(-4.370575,-5.9203506){\psarc[linecolor=black, linewidth=0.04, dimen=outer, arrowsize=0.05291667cm 2.0,arrowlength=1.4,arrowinset=0.0]{->}(12.81,-14.030132){8.54}{39.982376}{91.207924}}
\rput{54.82714}(-5.5712104,-17.878538){\psarc[linecolor=black, linewidth=0.04, dimen=outer, arrowsize=0.05291667cm 2.0,arrowlength=1.4,arrowinset=0.0]{->}(14.45,-14.310132){9.62}{0.0}{70.0}}
\rput{29.463062}(-3.677338,-6.914448){\psarc[linecolor=black, linewidth=0.04, dimen=outer, arrowsize=0.05291667cm 2.0,arrowlength=1.4,arrowinset=0.0]{->}(11.31,-10.450132){4.58}{30.977716}{89.12375}}
\psdots[linecolor=black, dotsize=0.16](20.02,-6.500132)
\psline[linecolor=black, linewidth=0.04, arrowsize=0.05291667cm 2.0,arrowlength=1.4,arrowinset=0.0]{<-}(18.54,-6.500132)(19.94,-6.500132)
\psline[linecolor=black, linewidth=0.04, arrowsize=0.05291667cm 2.0,arrowlength=1.4,arrowinset=0.0]{<-}(16.98,-6.520132)(18.38,-6.520132)
\psdots[linecolor=black, dotsize=0.16](16.82,-6.500132)
\psline[linecolor=black, linewidth=0.04, arrowsize=0.05291667cm 2.0,arrowlength=1.4,arrowinset=0.0]{<-}(15.42,-6.500132)(16.74,-6.500132)
\psdots[linecolor=black, dotsize=0.16](15.22,-6.500132)
\psline[linecolor=black, linewidth=0.04, arrowsize=0.05291667cm 2.0,arrowlength=1.4,arrowinset=0.0]{<-}(13.76,-6.500132)(15.18,-6.500132)
\psdots[linecolor=black, dotsize=0.16](13.62,-6.500132)
\psline[linecolor=black, linewidth=0.04, arrowsize=0.05291667cm 2.0,arrowlength=1.4,arrowinset=0.0]{<-}(12.18,-6.520132)(13.56,-6.500132)
\psdots[linecolor=black, dotsize=0.16](12.02,-6.500132)
\psline[linecolor=black, linewidth=0.04, arrowsize=0.05291667cm 2.0,arrowlength=1.4,arrowinset=0.0]{<-}(10.6,-6.500132)(12.0,-6.500132)
\psdots[linecolor=black, dotsize=0.16](10.42,-6.500132)
\psline[linecolor=black, linewidth=0.04, arrowsize=0.05291667cm 2.0,arrowlength=1.4,arrowinset=0.0]{<-}(8.96,-6.500132)(10.36,-6.500132)
\psdots[linecolor=black, dotsize=0.16](8.82,-6.500132)
\psdots[linecolor=black, dotsize=0.08](11.6,-3.760132)
\psdots[linecolor=black, dotsize=0.08](11.86,-4.120132)
\psdots[linecolor=black, dotsize=0.16](18.42,-6.500132)
\psline[linecolor=black, linewidth=0.04, arrowsize=0.05291667cm 2.0,arrowlength=1.4,arrowinset=0.0]{<-}(20.14,-6.520132)(21.54,-6.520132)
\end{pspicture}
}\caption{The dashed arcs mark the next cycle to be reversed.}
\label{fig:widget2}
 \end{figure}

 \begin{proof}
  Let $\{{n_k}:k<\oo\}$ be an increasing enumeration of $\{n<\oo:v_n\in N^-(v_0)\}$. Define $C_0=v_{0}v_{1}\dots v_{n_0}$, $C_1=v_0 v_{n_0} v_{n_0+1}\dots v_{n_{1}}$, and $C_{k+1}=v_0 v_{n_k} v_{n_k+1}\dots v_{n_{k+1}}$ in general (see Figure \ref{fig:widget2}). It is clear that $\mc C=\la C_k\ra_{k<\oo}\in \rs D$ is as required. 
  
 \end{proof}

 Both Observation \ref{widget1} and \ref{widget2} will play an important role later in the proof of the Main Theorem.

 \section{Part I: partition trees and uniformity}\label{sec:parttree}
 
 Our first goal is to show that one can make the strong components  of a tournament $D\in \mf T_\kappa$ \emph{uniform} in some sense by cycle reversions:

 \begin{theorem}\label{thm:parttree}Suppose $D\in \mf T_\kappa$. Then $V=V(D)$ can be partitioned into pieces $\{V_x:x\in \Gamma\}$ where $(\Gamma,\tri)$ is a linear order so that 
     \begin{enumerate}
      \item $D\uhr V_x$ is \emph{ $\kappa$-uniform} i.e. \[ |(N^{+}(v)\Delta N^{+}(v'))\cap V_x|<\kappa \textmd{  and }|(N^{-}(v)\Delta N^{-}(v'))\cap V_x|<\kappa\]  for all $v,v'\in V_x$, and 
      \item each arc $uv\in \arr{V_y V_x}$ such that $x\tri y$ is \emph{$\kappa$-reversible}.
     \end{enumerate}
   \end{theorem}
\begin{proof}
 
   Fix $D\in \mf T_{\kappa}$. We say that $\la V_x \ra_{x\in T}$ is a \emph{partition tree} if
     
     \begin{enumerate}
      \item $T$ is a downward closed subtree of sequences of ordinals,
      \item $x\perp y$ implies $V_x\cap V_y=\emptyset$,
      \item $x\sqsubset y$ (i.e. $x$ is an initial segment of $y$) implies $V_y\subset V_x$,
      \item\label{c:part} $V=\bigcup\{V_x:x\in T$ is terminal$\}$, and 
      \item any arc $a\in \arr{V_y V_x}$ so that $x\lex y$ is $\kappa$-reversible.
      
     \end{enumerate}
Note that if $V_{\emptyset}=V$ then $\la V_{\emptyset} \ra$ is a partition tree.
We say that a partition tree $\la V_x \ra_{x\in T}$  \emph{extends} $\la V'_y \ra_{y\in S}$ if $T$ is an end extension of $S$, and for all $y\in S$, $V_y=V'_y$. 

{\DD Our first lemma explains why we made this definition.} 

\begin{lemma}
If  $\la V_x \ra_{x\in T}$ is a maximal partition tree then $V_x$ is $\kappa$-uniform if $x$ is terminal in $T$.
    \end{lemma}
    
    We need the following definition: a \emph{$\kappa$-uniform ultrafilter} $\mc U$ on a set $I$ (of size $\kappa$) is a set of subsets of $I$ so that the following holds 
    \begin{enumerate}[(i)]
     \item $|W|=\kappa$ for all $W\in \mc U$,
     \item {\DD $I\in \mc U$ and $\mc U$ is closed under taking supersets, finite unions and finite intersections,}
     \item if $I=W\cup W'$ then $W\in \mc U$ or $W'\in \mc U$.
    \end{enumerate} The best to think of $\mc U$ as a measure of largeness among subsets of $I$. {\DD Recall that given $I$ and $W\subset I$ of size $\kappa$, there is a $\kappa$-uniform ultrafilter $\mc U$ on $I$ with $W\in \mc U$.}
    
    \begin{proof}
     Suppose that $V_z$ is not $\kappa$-uniform for some terminal $z\in T$ and so without loss of generality, we can suppose that $(N^+(v)\setm N^+(v'))\cap V_z$ has size $\kappa$ for some $v,v'\in V_z$.  Since $D\uhr V_z$ is a tournament, $(N^+(v)\setm N^+(v'))=(N^+(v)\cap N^-(v'))$. Take a $\kappa$-uniform ultrafilter $\mc U$ on $V_z$ containing $(N^+(v)\cap N^-(v'))\cap V_z$. Let $V_{z\smf 0}=\{u\in V_z:N^+(u)\cap V_z\in \mc U\}$ and $V_{z\smf 1}=\{u\in V_z:N^-(u)\cap V_z\in \mc U\}$. Note that $v\in V_{z\smf 0}$, $v'\in V_{z\smf 1}$ and $V_z=V_{z\smf 0}\dot\cup V_{z\smf 1}$. Figure \ref{fig:parttree} below shows how the block $V_z$ can be split to separate $v$ and $v'$.\begin{figure}[H]

   \centering
  \psscalebox{0.7 0.7} 
{
\begin{pspicture}(0,-5.17)(13.000001,3.17)
\definecolor{colour0}{rgb}{0.0,0.0,0.0}
\psframe[linecolor=black, linewidth=0.04, dimen=outer](2.0000007,3.13)(0.0,2.33)
\psframe[linecolor=black, linewidth=0.04, dimen=outer](4.0000005,3.13)(2.2000005,2.33)
\psframe[linecolor=black, linewidth=0.04, dimen=outer](9.000001,3.13)(7.200001,2.33)
\psframe[linecolor=black, linewidth=0.04, dimen=outer](5.0000005,3.13)(4.200001,2.33)
\psframe[linecolor=black, linewidth=0.04, dimen=outer](7.0000005,3.13)(5.200001,2.33)
\psline[linecolor=black, linewidth=0.04](4.200001,-2.67)(1.0000006,2.13)
\psline[linecolor=black, linewidth=0.04](4.200001,-2.67)(8.000001,2.13)
\psline[linecolor=black, linewidth=0.04](4.6000004,2.13)(6.6000004,0.33)
\psline[linecolor=black, linewidth=0.04](5.5600004,1.29)(6.200001,2.13)
\psline[linecolor=black, linewidth=0.04](2.2000005,0.33)(3.2000005,2.13)
\psdots[linecolor=black, dotsize=0.1](1.0000006,2.13)
\psdots[linecolor=black, dotsize=0.1](3.2000005,2.13)
\psdots[linecolor=black, dotsize=0.1](4.6000004,2.13)
\psdots[linecolor=black, dotsize=0.1](6.200001,2.13)
\psdots[linecolor=black, dotsize=0.1](8.000001,2.13)
\psline[linecolor=black, linewidth=0.04](4.200001,-2.67)(5.0000005,-3.87)
\psframe[linecolor=black, linewidth=0.04, dimen=outer](9.000001,-0.67)(0.0,-1.47)
\psline[linecolor=black, linewidth=0.04, linestyle=dashed, dash=0.17638889cm 0.10583334cm](4.0600004,-0.35)(4.3400006,-0.6001448)(3.9400005,-0.9687793)(4.2800007,-1.1794276)(3.9400005,-1.4954)(4.2800007,-1.73)
\psdots[linecolor=black, dotsize=0.2](2.5200007,2.65)
\psdots[linecolor=black, dotsize=0.2](7.640001,2.61)
\psbezier[linecolor=colour0, linewidth=0.04, arrowsize=0.05291667cm 2.0,arrowlength=1.4,arrowinset=0.0]{<-}(2.525636,2.7615988)(2.58411,4.024308)(7.7128396,3.9611106)(7.6543655,2.6984012343225547)
\psdots[linecolor=black, dotsize=0.2](7.640001,-1.11)
\psellipse[linecolor=black, linewidth=0.04, linestyle=dashed, dash=0.17638889cm 0.10583334cm, dimen=outer](8.68,-3.53)(2.26,0.64)
\psline[linecolor=colour0, linewidth=0.04, linestyle=dashed, dash=0.17638889cm 0.10583334cm, arrowsize=0.05291667cm 2.0,arrowlength=1.4,arrowinset=0.0]{->}(2.5200007,-1.15)(7.140001,-3.55)
\psline[linecolor=colour0, linewidth=0.04, linestyle=dashed, dash=0.17638889cm 0.10583334cm, arrowsize=0.05291667cm 2.0,arrowlength=1.4,arrowinset=0.0]{->}(7.180001,-3.55)(7.6000004,-1.25)
\psdots[linecolor=black, dotsize=0.2](2.5000007,-1.13)
\rput[bl](3.0000007,2.63){$w$}
\rput[bl](7.9200006,2.73){$w'$}
\rput[bl](9.9200006,-1){\huge{$V_z$}}
\rput[bl](1.7600006,-1.11){$w$}
\rput[bl](8.22,-1.05){$w'$}
\rput[bl](7.22,0.25){\huge{$V_{z\smf 1}$}}
\rput[bl](0.52,0.25){\huge{$V_{z\smf 0}$}}
\rput[bl](7.5200005,-3.7){$N^+(w)\cap N^-(w')\in \mc U$}
\rput[bl](2.0600007,3.87){$V_x$}
\rput[bl](7.7200007,3.77){$V_y$}
\end{pspicture}
}
 \caption{The partition tree with a backward arc}
\label{fig:parttree}

\end{figure}
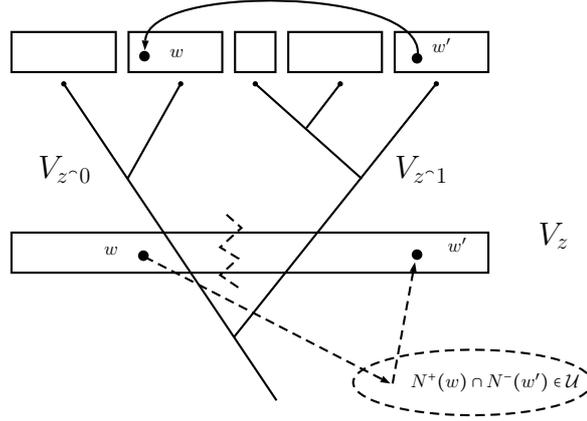
     
     Then  $S=T\cup\{z\smf 0,z\smf 1\}$ is an end extension of $T$ and we claim that $\la V_x\ra_{x\in S}$ is a partition tree properly extending $\la V_x \ra_{x\in T}$.  Indeed, every arc from {\DD $V_{z\smf 1}$ to $V_{z\smf 0}$} is $\kappa$-reversible:  Let $w' \in V_{z\smf 1}$ and $w \in V_{z\smf 0}$.  Then $(N^+(w)\cap N^-(w'))\cap V_z \in \mc U$, and so there are $\kappa$ many disjoint paths {\DD(of length 2)} from $w$ to $w'$.
     
     However, this contradicts the maximality of $\la V_x \ra_{x\in T}$.
    \end{proof}

In turn, the terminal nodes of a maximal partition tree give the desired decomposition, with $\tri$ defined by the lexicographic order of the corresponding tree.

{\DD Thus we only need to show that there are maximal partition trees. To this end, note that Zorn's lemma can be applied immediately once we proved the following.}


\begin{lemma}
 Any increasing chain of partition trees has an upper bound.
\end{lemma}
  \begin{proof}
   Given the increasing chain $\la V_x \ra_{x\in T_i}$ of partition trees with $i\in I$, we can first form $T^*=\bigcup\{T_i:i\in I\}$. $\la V_x \ra_{x\in T^*}$ satisfies each condition above except (\ref{c:part}). So, for each unbounded chain $b\subseteq T^*$ so that $\bigcap \{V_x:x\in b\}\neq \emptyset$ we add {\DD $\cup b$} to $T^*$ with $V_{\cup b}=\bigcap \{V_x:x\in b\}$. This defines a tree $T$ together with the sets  $\la V_x \ra_{x\in T}$. All conditions except (\ref{c:part}) are clearly satisfied.
   
   Lets check that (\ref{c:part}) holds for $\la V_x \ra_{x\in T}$. Given $v\in V$, we can find $x_i\in T_i$ terminal so that $v\in V_{x_i}$. In turn, $v\in \bigcap\{V_{x_i}:i\in I\}=V_{\cup b}$ for the chain $b$ generated by $\{x_i:i\in I\}\subseteq T^*$.

  \end{proof}


  This finishes the proof of the theorem.

\end{proof}

\begin{cor}\label{cor:first}
 If $D\in \mf T_\kappa$ then there is $\mc C\in \rs D$ so that strong components of $\rev{D}{\mc C}$ are $\kappa$-uniform.
\end{cor}
\begin{proof}
 Suppose that $\{V_x:x\in \Gamma\}$ is as in Theorem \ref{thm:parttree}. Let $F=\{uv:uv\in \arr{V_y V_x}$ such that $x\tri y\}$. Each arc in $F$ is $\kappa$-reversible and $|F|\leq \kappa$ so Lemma \ref{lm:kapparev} applies: there is $\mc C\in \rs D$ so that reversing $\mc C$ reverses exactly the arcs in $F$. Now, each strong component of $\rev{D}{\mc C}$ must be contained in some $V_x$, and each such set is $\kappa$-uniform.
\end{proof}

 \section{Part II: tournaments with small in-or out-degrees}

 In this section, we wish to better understand  a restricted class of $\kappa$-uniform tournaments: we suppose that either all in-degrees or all out-degrees are less than $\kappa$. {\DD Let $\Delta^+(D)$ denote the minimal cardinal $\nu$ so that $d^+(v)<\nu$ for all $v\in V(D)$, and define $\Delta^-(D)$ similarly.} So, we write $\Delta^+(D)\leq \kappa$ iff all out-degrees in $D$ are $<\kappa$ and  $\Delta^-(D)\leq \kappa$ iff all in-degrees in $D$ are $<\kappa$.
 
 A canonical example of a tournament with $\Delta^+(D)\leq \kappa$ is defined on the ordinal $\kappa$ with arcs $\{\xi\zeta:\zeta<\xi<\kappa\}$. Now, if each vertex $\xi$ is blown up into an arbitrary tournament of size $<\kappa$ then the resulting tournament still satisfies $\Delta^+(D)\leq \kappa$. {\DD Our second structural result says  that,} after appropriate cycle reversions, any tournament $D$ with $\Delta^+(D)\leq \kappa$ looks like this.
 
 \medskip
 
 We start with a lemma which will also play a key role in later sections. The lemma says that if we have a small set of vertices $W$ with (say) all out-degrees small then $W$ can be enlarged to a still small $W^*$ such that, after reversing a small number of cycles, all arcs point into $W^*$ from $V\setm W^*$.
 
 \begin{lemma}\label{lm:degreecontroll}
  Suppose that $D\in \mf T_\kappa$ and $W\in [V]^{<\lambda}$ for some $\aleph_0\leq \lambda\leq \kappa$ with $\cf(\lambda)=\lambda$. Also, fix a set $F$ of $<\lambda$ many edges.
  
  If $d^+(v)<\lambda$ for all $v\in W$ then there is $\mc C\in \rs D$ and $W\subseteq W^*$ so that $|W^*|<\lambda$ and all arcs point into $W^*$ from $V\setm W^*$ in $\rev{D}{\mc C}$. 
  
   If $d^-(v)<\lambda$ for all $v\in W$ then there is $\mc C\in \rs D$ and $W\subseteq W^*$ so that $|W^*|<\lambda$ and all arcs point from $W^*$ to $V\setm W^*$ in $\rev{D}{\mc C}$. 
   
   Furthermore, if $\lambda$ is uncountable then we can also ensure that $|\mc C|<\lambda$ and $E(\mc C)\cap F=\emptyset$.
 \end{lemma}
 
{\DD  We let 
 
 \begin{enumerate}[(i)]
  \item $N^{+0}_D(v)=\{v\}$, $N^{+1}_D(v)=\{v\}\cup N^+_D(v)$,  
  \item in general, for $\ell\geq 1$, let $$N^{+(\ell+1)}_D(v)=N^{+\ell}_D(v)\cup \bigcup\{N^+_D(u):u\in N^{+\ell}_D(v)\},$$
  \item $N^{+\oo}_D(v)=\bigcup_{\ell<\oo}N^{+\ell}_D(v)$.
 \end{enumerate}}

 We let $N^{+\ell}(W)=\bigcup \{N^{+\ell}(v):v\in W\}$ for $\ell\leq \oo$.

\begin{proof}Suppose $d^+(v)<\lambda$ for all $v\in W$.  First, we define a sequence $\mc C_n\in \rs D$ so that 
 \begin{enumerate}[(a)]
  \item $|\mc C_n|<\lambda$, $E(\mc C_n)\cap E(\mc C_k)=\emptyset$ for all $n<k<\oo$, and 
  \item $d^+_{\rev{D}{\mc C_{\leq n}}}(v)<\lambda$ for all $v\in N^{+n}_{\rev{D}{\mc C_{\leq n}}}(W)$.
 \end{enumerate}
 
 {\DD Here, we let $\mc C_{\leq n}$ denote $\mc C_0\smf \mc C_1\smf \dots \smf \mc C_n$.}
 
 Clearly, $\mc C_0=\emptyset$ is a good choice to start with. In order to continue the induction, suppose that $u\in N^{+n}_{\rev{D}{\mc C_{\leq n}}}(W)$ and $v\in N^+(u)$ with $d^+(v)\geq \lambda$.  Then since $|N^{+n}_{\rev{D}{\mc C_{\leq n}}}(W)|<\lambda$, $v$ has at least $\lambda$ out-neighbors in $V \setm (N^{+n}_{\rev{D}{\mc C_{\leq n}}}(W))$, each of which, by construction, are in-neighbors of $u$.  Hence the arc $uv$ is $\lambda$-reversible.
 
 \begin{figure}[H]
  \psscalebox{0.8} 
{
\begin{pspicture}(0,-3.3434713)(12.624027,3.3434713)
\psellipse[linecolor=black, linewidth=0.04, dimen=outer](1.1307693,0.033451773)(0.9,1.7)
\psellipse[linecolor=black, linewidth=0.04, dimen=outer](2.8846154,-0.8050098)(0.3,0.6)
\psline[linecolor=black, linewidth=0.04](2.8460114,-0.235779)(1.1846154,-0.81280744)(2.8615384,-1.3742405)
\psdots[linecolor=black, dotsize=0.16](2.9384615,-0.86654824)
\psellipse[linecolor=black, linewidth=0.04, dimen=outer](4.1615386,-0.20500976)(0.26923078,1.2307693)
\psline[linecolor=black, linewidth=0.04](4.0923076,0.964221)(2.9692307,-0.8511636)(4.1384616,-1.4050097)
\psline[linecolor=black, linewidth=0.04, arrowsize=0.05291667cm 2.0,arrowlength=1.4,arrowinset=0.0]{->}(2.9846153,-0.8511636)(4.1846156,-0.17424053)
\rput{-121.38128}(6.381867,0.14191893){\psarc[linecolor=black, linewidth=0.04, linestyle=dashed, dash=0.17638889cm 0.10583334cm, dimen=outer, arrowsize=0.05291667cm 2.0,arrowlength=1.4,arrowinset=0.0]{->}(3.2307692,-1.7203944){2.3307693}{188.18465}{270.0}}
\rput[bl](2.1846154,-2.0819328){$d^+(u)<\lambda$}
\rput[bl](3.4307692,1.2565287){$d^+(v)\geq \lambda$}
\rput[bl](0.0,1.564221){\Large{$W$}}
\psellipse[linecolor=black, linewidth=0.04, dimen=outer](9.161539,-0.15116362)(0.9,1.7)
\rput[bl](11.184615,2.2){\Large{$W^*$}}
\psbezier[linecolor=black, linewidth=0.04](7.8615384,2.2719133)(7.3030047,1.4424314)(8.084115,-2.1396394)(8.984615,-2.3896251502403847)(9.885116,-2.6396108)(10.51839,-0.77295315)(11.030769,0.5180672)(11.543149,1.8090875)(13.192808,2.2680814)(12.307693,2.7334518)(11.422576,3.198822)(8.420073,3.1013951)(7.8615384,2.2719133)
\rput[bl](8.2,1.4257594){\Large{$W$}}
\psdots[linecolor=black, dotsize=0.16](1.1076924,-0.8050098)
\psline[linecolor=black, linewidth=0.04, arrowsize=0.05291667cm 2.0,arrowlength=1.4,arrowinset=0.0]{->}(1.1230769,-0.8050098)(2.876923,-0.86654824)
\rput[bl](0.93846154,-1.235779){$u$}
\rput[bl](2.8153846,-1.2050098){$v$}
\psdots[linecolor=black, dotsize=0.16](9.2153845,-1.0050098)
\psline[linecolor=black, linewidth=0.04, arrowsize=0.05291667cm 2.0,arrowlength=1.4,arrowinset=0.0]{<-}(9.307693,-1.0203943)(11.061539,-1.035779)
\rput[bl](9.046154,-1.435779){$u$}
\rput[bl](10.923077,-1.4050097){$v$}
\psdots[linecolor=black, dotsize=0.16](11.076923,-1.0511636)
\rput[bl](0.52307695,-3.4){\Large{$D$}}
\rput[bl](8.8,-3.4){\Large{$\rev{D}{\mc C}$}}
\rput{-18.061342}(-0.065531746,3.423812){\psarc[linecolor=black, linewidth=0.04, dimen=outer, arrowsize=0.05291667cm 2.0,arrowlength=1.4,arrowinset=0.0]{->}(10.7384615,1.9180672){1.4923077}{127.57948}{185.36653}}
\rput{-89.87703}(11.135775,10.349681){\psarc[linecolor=black, linewidth=0.04, dimen=outer, arrowsize=0.05291667cm 2.0,arrowlength=1.4,arrowinset=0.0]{->}(10.753846,-0.40500978){1.4923077}{127.57948}{185.36653}}
\rput{-18.061342}(0.29152307,3.19346){\psarc[linecolor=black, linewidth=0.04, dimen=outer, arrowsize=0.05291667cm 2.0,arrowlength=1.4,arrowinset=0.0]{<-}(10.192307,0.6796056){1.0846153}{127.57948}{185.36653}}
\rput{-18.061342}(-0.15710214,0.77073234){\psarc[linecolor=black, linewidth=0.04, dimen=outer, arrowsize=0.05291667cm 2.0,arrowlength=1.4,arrowinset=0.0]{<-}(2.3461537,0.8796056){1.0846153}{127.57948}{185.36653}}
\end{pspicture}
}
\vspace{0.5cm}
\caption{Reversing the $uv$ arcs when $d^+(v)$ is too large.}
\label{fig:degcontroll}

 \end{figure}
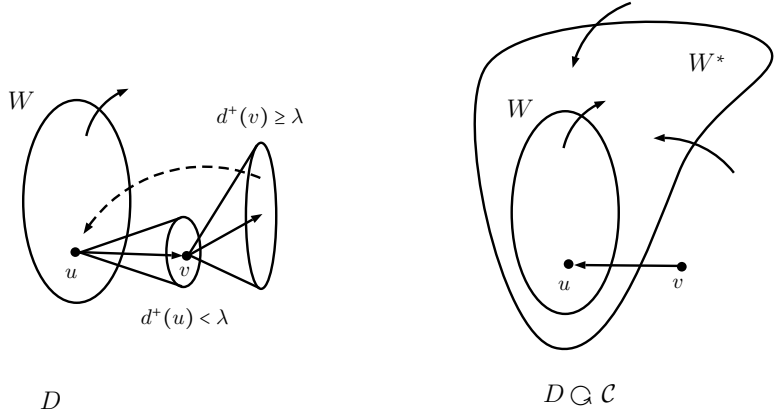
 
  Furthermore, there are less than $\lambda$ many such $uv$ arcs (since $v\in N^{+(n+1)}_{\rev{D}{\mc C_{\leq n}}}(W)$ and $|N^{+(n+1)}_{\rev{D}{\mc C_{\leq n}}}(W)|<\lambda$). Finally, there are $<\lambda$ many arcs from $\mc C_{\leq n}$ to avoid. Hence we can proceed by applying Lemma \ref{lm:kapparev} to find $\mc C_{n+1}$ which reverses all these $uv$ arcs while avoiding $\mc C_{\leq n}$ (and arcs in $F$, if needed).
 
 Now $W^*=N^{+\oo}_{\rev{D}{\mc C_{<\oo}}}(W)$ will have size $<\lambda$ and $|\mc C_{<\oo}|<\lambda$ if $\lambda>\aleph_0$ and we are done.
 
 If $\lambda=\aleph_0$ then $N^{+\oo}_{\rev{D}{\mc C_{<\oo}}}(W)$ can have size $\aleph_0$ but all out-degrees are finite in $N^{+\oo}_{\rev{D}{\mc C_{<\oo}}}(W)$. So, it suffices to prove the following:
 
  \begin{lemma}\label{lm:findeg}
          Suppose that $D$ is a tournament and $W\subseteq V$ is finite. If {\DD $\Delta^+(D)\leq \aleph_0$} then there is $\mc C\in \rs D$ so that $N^{+\oo}_{\rev{D}{\mc C}}(W)$ is finite and {\DD $\Delta^+(\rev{D}{\mc C})\leq \aleph_0$} still holds.

          If {\DD$\Delta^-(D)\leq \aleph_0$} then there is $\mc C\in \rs D$ so that $N^{-\oo}_{\rev{D}{\mc C}}(W)$ is finite and {\DD$\Delta^-(\rev{D}{\mc C})\leq \aleph_0$} still holds. 
          
         \end{lemma}
         
         
\begin{proof} First, suppose that $W=\{v\}$. Suppose $\Delta^+(D)=\aleph_0$ and find $\mc C\in \rs D$ which minimizes $N^{+}_{\rev{D}{\mc C}}(v)$ while keeping {\DD $\Delta^+(\rev{D}{\mc C})\leq \aleph_0$}. We claim that $N^{+\oo}_{\rev{D}{\mc C}}(v)$ is finite.
 
 Otherwise, K\H onig's lemma implies that there are $v=v_0,v_1,v_2\dots$ so that $v_{n}v_{n+1}\in A(\rev{D}{\mc C})$. (To apply K\H onig's lemma we need {\DD$\Delta^+(D)\leq \aleph_0$}, the assumption that $v$ has finite degree is not enough.) Furthermore,  note that $v_nv_0\in A(\rev{D}{\mc C})$ for all but finitely many $n$ since $d^+_{\rev{D}{\mc C}}(v)$ is finite. So, we can apply Observation \ref{widget2} to find $\la C_n\ra_{n\in \oo}\in \rs{\rev{D}{\mc C}}$ so that $v_{n+1}v_{n}\in A(\rev{D}{\mc C^*})$ for all $n\in \oo$ where $\mc C^*=\mc C\smf\la C_n\ra_{n\in \oo}$ and no other arcs from $\rev{D}{\mc C}$ are changed.

%
%
%

It is clear that $d^+_{\rev{D}{\mc C^*}}(v)=d^+_{\rev{D}{\mc C}}(v)-1$ while $d^+_{\rev{D}{\mc C^*}}(v')=d^+_{\rev{D}{\mc C}}(v')$ for all $v'\in V\setm \{v\}$. This contradicts the choice of $\mc C$.
 
 Now, proceed by induction on $|W|$. Let $v\in W$ and find $\mc C_0$ so that  $N=N^{+\oo}_{\rev{D}{\mc C_0}}(W\setm\{v\})$ is finite and each vertex still has finite out-degree in $\rev{D}{\mc C_0}$. If $v\in N$ then we are done. Otherwise, apply the singleton case in $D_1=(\rev{D}{\mc C_0})\uhr( V\setm N)$ for $v$. Let $\mc C_1\in \rs{D_1}$  be such that $N'=N^{+\oo}_{\rev{D_1}{\mc C_1}}(v)$ is finite and still each vertex has finite in-degree. Now $N^{+\oo}_{\rev{D}{(\mc C_0\smf \mc C_1)}}(W)\subseteq N\cup N'$ is finite. Clearly,{\DD $\Delta^+(\rev{D}{(\mc C_0\smf \mc C_1)})\leq \aleph_0$.}
 \end{proof}
 
 
\end{proof}

 Now, the main result of this section is the following:

 \begin{theorem}\label{thm:smalldegree}
  Suppose that $D\in \mf T_{\kappa}$ and {\DD$\mu=\cf(\kappa)$}.
  
  If {\DD $\Delta^+(D)\leq \kappa$} then there is  $\mc C\in \rs D$ and a partition $\bigsqcup\{V_\xi:\xi<\mu\}$ of $V(D)$ so that $|V_\xi|<\kappa$ and $uv\in A(\rev{D}{\mc C})$ for all $u\in V_\xi,v\in V_\zeta$ with $\zeta<\xi$. 
  
   If {\DD$\Delta^-(D)\leq \kappa$} then there is  $\mc C\in \rs D$ and a partition $\bigsqcup\{V_\xi:\xi<\mu\}$ of $V(D)$ so that $|V_\xi|<\kappa$ and $uv\in A(\rev{D}{\mc C})$ for all $v\in V_\xi,u\in V_\zeta$ with $\zeta<\xi$. 
 \end{theorem}
\begin{proof} We prove for {\DD$\Delta^+(D)\leq \kappa$.}

 First, if $\kappa=\aleph_0$ then we can simply apply Lemma \ref{lm:degreecontroll} to define the finite blocks inductively in $\oo$ steps. Let $\prec$ be an enumeration of $V$ of type $\oo$. Given finite $V_0,V_1\dots V_n$ and (possibly infinite) $\mc C_0\dots \mc C_n$, take $v_{n+1}=\min_{\prec}V\setm (V_0\cup\dots \cup V_n)$ and apply Lemma \ref{lm:degreecontroll} to $\{v_{n+1}\}$ in $\rev{D}{\mc C_{\leq n}}\uhr V\setm (V_0\cup\dots \cup V_n)$. Now we have $\mc C_{n+1}$ and $V_{n+1}\supseteq \{v_{n+1}\}$ so that all arcs point into $V_{n+1}$ from $V\setm (V_0\cup\dots \cup V_{n+1})$. It is clear that $\mc C=\la \mc C_n\ra_{n\in \oo}\in \rs D$ is as desired.
 
 Second, if $\kappa>\aleph_0$ is regular, then we don't need to reverse any cycles. Indeed, if $W$ has size $<\kappa$ then $N^{+\oo}(W)$ still has size $<\kappa$. So again, an easy induction gives the result.
 
 Finally, suppose that $\kappa>\cf(\kappa)=\mu$; let $\la\kappa_\xi\ra_{\xi<\mu}$ be a continuous, increasing, and  cofinal sequence of cardinals in $\kappa$.
      
      Write $V=\bigsqcup\{W_\xi:\xi<\mu\}$ where $W_0=\{v\in V:d^+(v)\leq \kappa_0\}$ and $W_{\xi}=\{v\in V:\sup\{\kappa_\zeta:\zeta<\xi\}<d^+(v)\leq \kappa_\xi\}$. 
      
      \begin{lemma}
       $|W_\xi|\leq \kappa_\xi^+$ for all $\xi<\mu$.
      \end{lemma}
\begin{proof}
Let us recall the following weak consequence of A. Hajnal's Set Mapping Theorem \cite{setmapping}: if $\lambda$ is an infinite cardinal and $f:X\to \mc P(X)$ so that $|X|\geq \lambda^{++}$ but $|f(x)|\leq \lambda$ for all $x\in X$ then there is $x\neq y\in X$  so that $x\notin f(y)$ and $y\notin f(x)$. Indeed, find a $Y\subseteq X$ of size $\leq \lambda^+$ so that $y\in Y$ implies $f(y)\subseteq Y$. Now, if $x\in X\setm Y$ and $y\in Y\setm f(x)$ then $x\notin f(y)$ and $y\notin f(x)$ as desired.

Now, if  $W_\xi$ has size at least $\kappa_\xi^{++}$ then we can consider the map $f:W_\xi\to \mc P(W_\xi)$ defined by $f(v)=N^+(v)\cap W_\xi$ for $v\in W_\xi$. $|f(v)|\leq \kappa_\xi$ by assumption so there must be some  $x\neq y\in W_\xi$ so that $x\notin N^+(y)$ and $y\notin N^+(x)$. However, this contradicts that $D$ was a tournament i.e. either $xy$ or $yx$ is an arc. 
     \end{proof}
 
 Now, define $W^*_\xi$ and $\mc C_\xi$ as follows: apply Lemma \ref{lm:degreecontroll} to $W_0$ with $\lambda=\kappa_0^{++}$ to find   $W^*_0$ of size $\leq \kappa_0^+$  which contains $W_0$ so that all arcs point into $W^*_0$ from $V\setm W^*_0$ after reversing $\mc C_0$ (at most $\kappa_0^+$ many cycles). In general, let $W^*_{<\xi}=\bigcup\{W^*_\zeta:\zeta<\xi\}$ and if $W_\xi\setm W^*_{<\xi}$ is not empty then we apply Lemma \ref{lm:degreecontroll} to  $W_\xi\setm W^*_{<\xi}$ in $V\setm W^*_{<\xi}$ with $\lambda=\kappa_\xi^{++}$ and $F=\bigcup\{A(\mc C_\zeta):\zeta<\xi\}$. Since $\sup_{\zeta<\xi}\kappa_\zeta^+<\kappa_\xi^{++}$, it is not an issue to avoid all previously used arcs when finding $\mc C_\xi$ and $W^*_\xi\supseteq W_\xi\setm W^*_{<\xi} $. {\DD So, $V_\xi=W^*_\xi$ for $\xi<\mu$ is as desired.}

\end{proof}

In particular, we get the next corollary:

\begin{cor}\label{cor:second}
  Suppose that $D\in \mf T_{\kappa}$ and {\DD $\Delta^+(D)\leq \kappa$  or $\Delta^-(D)\leq \kappa$}. Then there is  $\mc C\in \rs D$ so that strong components of  $\rev{D}{\mc C}$ have size $<\kappa$.
\end{cor}

 \section{Part III: uniform tournaments in general}\label{sec:unif}
 
 Now, {\DD we return to the class $\mf{UT}_\kappa$ i.e.} arbitrary $\kappa$-uniform tournaments: our aim is to show that, after appropriate cycle reversions, in each  strong components of $D$ either the in-degrees are $<\kappa$ or the out-degrees are $<\kappa$. {\DD Once we achieved this,} we can make each strong component small in size by Part II.

 \begin{theorem}\label{thm:smallcomp}
  Suppose that $D\in \mf{UT}_\kappa$. Then there is $\mc C\in \rs D$ and $V=V^+\cup V^-$ so that
  \begin{enumerate}
   \item $\Delta^-((\rev{D}{\mc C})\uhr V^+)\leq \kappa$,
   \item $\Delta^+((\rev{D}{\mc C})\uhr V^-)\leq \kappa$, and 
   \item each strong component of $\rev{D}{\mc C}$ is contained either in $V^+$ or $V^-$.
  \end{enumerate}

 \end{theorem}
 
 {\DD It will be clear immediately  that (1) and (2) are very easily satisfied while we need to work for (3) more.}
 
\begin{proof} First, we need the following:

 \begin{lemma}\label{lm:unifpart}
  Suppose that $D\in \mf {UT}_{\kappa}$. Then there is a partition $V=V^+\cup V^-$ so that in-degrees are $<\kappa$ in $D\uhr V^+$ and  out-degrees are $<\kappa$ in $D\uhr V^-$. Moreover, either
 \begin{enumerate}[(a)]
   \item any arc $a\in \arr{V^-V^+}$ is $\kappa$-reversible, or 
  \item there is $W\in [V]^{<\kappa}$ so that any arc $a\in \arr{V^+V^-}$ meets $W$.
 
 \end{enumerate}
 \end{lemma}
\begin{proof}
 Pick any ${v_0}\in V$ and let $V^+=\{{v_0}\}\cup N^+({v_0})$ and $V^-=N^-({v_0})$. Then $\Delta^-(D\uhr V^+)=\kappa$ and $\Delta^+(D\uhr V^-)=\kappa$ by $\kappa$-uniformity:  If (say) $u\in D\uhr V^+$, then $|N^{+}({v_0})\Delta N^{+}(u)|<\kappa$, and since $D$ is a tournament, this means $|N^{+}({v_0})\cap N^{-}(u)|<\kappa$ and  $|V^{+}\cap N^{-}(u)|<\kappa$.
 
 If (b) fails then we can find $\kappa$ many vertex disjoint arcs $u_\xi v_\xi\in\arr{V^+V^-}$. Suppose that $uv\in \arr{V^-V^+}$. All but $<\kappa$ many $u_\xi$ are contained in $N^+(v)$ and all but $<\kappa$ many $v_\xi$ are in $N^-(u)$. So, $vu_\xi v_\xi u$ is a path for $\kappa$ many $\xi$. {\DD In turn, (a) holds.}
 
\end{proof}

Now, fix $V=V^+\cup V^-$ as in  Lemma \ref{lm:unifpart}. 

If (a) holds, then we can simply reverse all the arcs from  $\arr{V^-V^+}$ without changing $D\uhr V^+$ or $D\uhr V^-$ using Lemma \ref{lm:kapparev}. Now, strong components are contained in either $V^+$ or $V^-$, and Theorem \ref{thm:smallcomp} is proved.
  
 Suppose (b) holds and $W\in [V]^{<\kappa}$ pins down all arcs in $\arr{V^+V^-}$. If $\cf(\kappa)=\kappa$ then we can apply Lemma \ref{lm:degreecontroll} in $D\uhr (V^-\cup W)$ for the set $W$ with $\lambda=\kappa$ to find $W\subseteq W^*\subseteq W\cup V^-$ so that $W^*$ has size $<\kappa$ and all arcs point into $W^*$ from $V^-\setm W^*$ after reversing $<\kappa$ many cycles.
 
 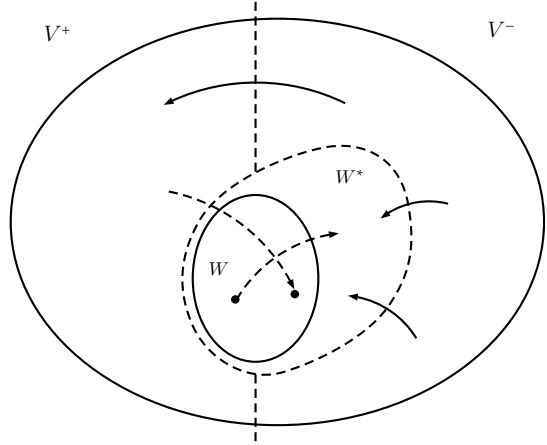
\begin{figure}[H]
  \psscalebox{0.7} 
{
\begin{pspicture}(0,-4.323752)(10.06,4.323752)
\psellipse[linecolor=black, linewidth=0.04, dimen=outer](5.03,0.15375184)(5.03,3.87)
\psline[linecolor=black, linewidth=0.04, linestyle=dashed, dash=0.17638889cm 0.10583334cm](4.62,4.323752)(4.62,1.1037518)
\rput[bl](0.64,3.583752){\Large{$V^+$}}
\rput[bl](8.98,3.6637518){\Large{$V^-$}}
\psellipse[linecolor=black, linewidth=0.04, dimen=outer](4.62,-0.91624814)(1.2,1.6)
\psbezier[linecolor=black, linewidth=0.04, linestyle=dashed, dash=0.17638889cm 0.10583334cm](6.04,1.5837518)(5.403813,1.5068214)(4.036139,0.971995)(3.58,0.2037518310546875)(3.123861,-0.56449133)(3.183861,-1.268005)(3.46,-1.8562481)(3.7361388,-2.4444914)(4.458476,-2.8909693)(5.12,-2.696248)(5.781524,-2.501527)(6.714357,-2.0304055)(7.12,-1.4362482)(7.525643,-0.8420909)(7.6369543,-0.31279394)(7.48,0.52375185)(7.3230453,1.3602976)(6.676187,1.6606823)(6.04,1.5837518)
\rput{-152.59897}(9.189832,0.28783074){\psarc[linecolor=black, linewidth=0.04, dimen=outer, arrowsize=0.05291667cm 2.0,arrowlength=1.4,arrowinset=0.0]{->}(4.63,-0.97624815){3.77}{216.25806}{270.0}}
\rput{-119.84238}(11.941553,1.7238808){\psarc[linecolor=black, linewidth=0.04, linestyle=dashed, dash=0.17638889cm 0.10583334cm, dimen=outer, arrowsize=0.05291667cm 2.0,arrowlength=1.4,arrowinset=0.0]{<-}(6.47,-2.5962481){2.55}{216.25806}{270.0}}
\rput{-191.28943}(4.140171,-5.461706){\psarc[linecolor=black, linewidth=0.04, linestyle=dashed, dash=0.17638889cm 0.10583334cm, dimen=outer, arrowsize=0.05291667cm 2.0,arrowlength=1.4,arrowinset=0.0]{<-}(2.34,-2.5262482){3.32}{216.25806}{270.0}}
\rput[bl](3.74,-0.85624814){$W$}
\rput[bl](6.12,0.8837518){$W^*$}
\rput{-185.26646}(12.106151,-6.3892703){\psarc[linecolor=black, linewidth=0.04, dimen=outer, arrowsize=0.05291667cm 2.0,arrowlength=1.4,arrowinset=0.0]{->}(6.2,-2.916248){1.68}{216.25806}{270.0}}
\rput{-140.63493}(14.540642,3.4088078){\psarc[linecolor=black, linewidth=0.04, dimen=outer, arrowsize=0.05291667cm 2.0,arrowlength=1.4,arrowinset=0.0]{->}(7.88,-0.89624816){1.44}{216.25806}{270.0}}
\psdots[linecolor=black, dotsize=0.16](5.36,-1.2162482)
\psdots[linecolor=black, dotsize=0.16](4.24,-1.3162482)
\psline[linecolor=black, linewidth=0.04, linestyle=dashed, dash=0.17638889cm 0.10583334cm](4.62,-2.7362483)(4.62,-4)
\end{pspicture}
}

\caption{The modified partition from $V^+$ and $V^-$.}
\label{fig:finalproof}
 \end{figure}

 Now, note that all arcs point into $V^+\cup W^*$ from $V^-\setm W^*$. So, each strong component  is either in $V^+\cup W^*$ or $V^-\setm W^*$. In the first set, all in-degrees are $<\kappa$ and in the second set, all out-degrees are $<\kappa$. This is true because $|W^*|<\kappa$ and we only reversed $<\kappa$ many cycles so no in-degree or out-degree was raised to $\kappa$.
 
 Finally, suppose that $\kappa>\cf(\kappa)=\mu$. Apply Theorem \ref{thm:smalldegree} first to $D\uhr V^+\cup W$ and then to $D\uhr V^-\cup W$ to get $\la V^+_\xi\ra_{\xi<\mu}$ and $\la V^-_\xi \ra_{\xi<\mu}$. Note that in the second reversion, we might make some arcs of $W$ point backward with respect to the $\la V^+_\xi\ra_{\xi<\mu}$ decomposition but this will not really matter.


We distinguish two cases.

\textbf{Case 1.} {\DD Suppose that, for each $\xi<\mu$ and $\lambda<\kappa$, there is $w\in (V^+\cup W)\setm V^+_{<\xi}$ so that $|N^+(w)\cap V^-|\geq \lambda$.}
 
 \begin{claim}
  Each arc $uv\in \arr{V^-V^+}$ is $\kappa$-reversible.
 \end{claim}

\begin{proof} Suppose that we fix an arc $uv\in\arr{V^-V^+}$ and  $\lambda<\kappa$. Let $\xi<\mu$ be large enough so that $|V^+_\xi|\geq \max\{\lambda, |W|^+$\}, $v\in V^+_{<\xi},u\in V^-_{<\xi}$. Now select $w\in (V^+\cup W)\setm V^+_{<\xi}$  so that $(N^+(w)\cap V^-)\setm V^-_{<\xi}$ has size at least $\lambda$.

The above choices make sure that $vxw y u$ is a path for all $x\in V^+_\xi\setm W$ and $y\in (N^+(w)\cap V^-)\setm V^-_{<\xi}$; see Figure \ref{fig:paths}. Distinct choices of $x$ and $y$ give edge-disjoint paths from $v$ to $u$ so we can find $\lambda$ many of these.  Finally, since $uv$ is $\lambda$-reversible for cofinally many $\lambda<\kappa$, then $uv$ is also $\kappa$-reversible.
\end{proof}

\begin{figure}[H]
 \psscalebox{0.7} 
{
\begin{pspicture}(0,-4.16)(10.96,4.16)
\psellipse[linecolor=black, linewidth=0.04, dimen=outer](6.06,-0.02)(4.9,3.5)
\psline[linecolor=black, linewidth=0.04, linestyle=dashed, dash=0.17638889cm 0.10583334cm](6.0,3.86)(6.0,-4.16)
\psline[linecolor=black, linewidth=0.03](5.14,3.38)(4.74,1.96)(5.2,0.4)(4.72,-0.9)(4.92,-1.98)(4.7,-3.34)
\psline[linecolor=black, linewidth=0.03](7.2,3.36)(7.62,1.66)(7.08,-0.38)(7.48,-1.7)(7.2,-3.4)
\rput[bl](6.5,3.8){\Large{$W$}}
\rput[bl](10.08,2.98){\Large{$V^-$}}
\rput[bl](1.56,3.04){\Large{$V^+$}}
\rput{-126.931145}(11.365342,4.31508){\psarc[linecolor=black, linewidth=0.04, dimen=outer, arrowsize=0.05291667cm 2.0,arrowlength=1.4,arrowinset=0.0]{<-}(6.76,-0.68){2.18}{0.0}{73.94481}}
\psdots[linecolor=black, dotsize=0.2](8.08,-2.42)
\psdots[linecolor=black, dotsize=0.2](5.36,-2.36)
\rput[bl](8.22,-2.06){$u$}
\rput[bl](5.24,-1.9){$v$}
\rput{-150.70508}(17.708078,-8.371724){\psarc[linecolor=black, linewidth=0.04, linestyle=dashed, dash=0.17638889cm 0.10583334cm, dimen=outer](7.76,-6.5){6.06}{211.97797}{270.0}}
\rput[bl](8.8,-1.66){\Large{$V^-_{<\xi}$}}
\rput[bl](2.24,-1.76){\Large{$V^+_{<\xi}$}}
\rput{-150.70508}(9.422273,-7.7373457){\psarc[linecolor=black, linewidth=0.04, linestyle=dashed, dash=0.17638889cm 0.10583334cm, dimen=outer](3.7,-5.1){4.84}{211.97797}{270.0}}
\rput{-150.70508}(8.644491,-5.2206306){\psarc[linecolor=black, linewidth=0.04, linestyle=dashed, dash=0.17638889cm 0.10583334cm, dimen=outer](3.64,-3.74){4.84}{211.97797}{270.0}}
\psdots[linecolor=black, dotsize=0.2](2.7,0.32)
\psdots[linecolor=black, dotsize=0.2](3.44,1.76)
\psline[linecolor=black, linewidth=0.04, arrowsize=0.05291667cm 2.0,arrowlength=1.4,arrowinset=0.0]{->}(2.72,0.34)(3.38,1.58)
\psline[linecolor=black, linewidth=0.04, arrowsize=0.05291667cm 2.0,arrowlength=1.4,arrowinset=0.0]{->}(5.38,-2.42)(2.84,0.18)
\rput[bl](1.94,0.08){$x$}
\rput[bl](3.12,2.12){$w$}
\psellipse[linecolor=black, linewidth=0.04, linestyle=dashed, dash=0.17638889cm 0.10583334cm, dimen=outer](6.82,1.54)(0.46,0.98)
\psline[linecolor=black, linewidth=0.04, arrowsize=0.05291667cm 2.0,arrowlength=1.4,arrowinset=0.0]{->}(3.4524372,1.7423453)(6.687563,1.4976547)
\psline[linecolor=black, linewidth=0.04, arrowsize=0.05291667cm 2.0,arrowlength=1.4,arrowinset=0.0]{->}(6.86,1.48)(7.98,-2.24)
\rput[bl](0.0,-0.32){\Large{$V^+_{\xi}$}}
\psdots[linecolor=black, dotsize=0.2](6.84,1.5)
\rput[bl](6.76,1.84){$y$}
\end{pspicture}
}\caption{Constructing the $v\to u$ paths in Case 1.}
\label{fig:paths}
\end{figure}
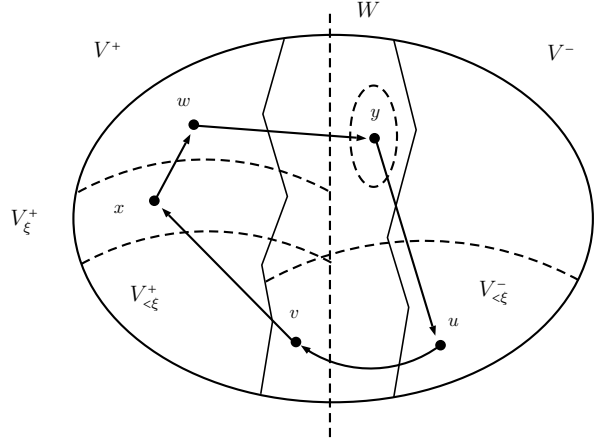

 Now, we can again make each arc between $V^+$ and $V^-$ point into $V^-$ without changing anything in $D\uhr V^+$ or $D\uhr V^-$ via Lemma \ref{lm:kapparev}. So strong components will be contained in $V^+$ or $V^-$.

 \textbf{Case 2.} If Case 1 fails, then there is some $\xi_0<\mu$ and $\lambda<\kappa$ so that $|N^+(v)\cap V^-|\leq \lambda$ for all $v\in (V^+\cup W)\setm V^+_{<\xi_0}$. We can suppose that $\lambda$ is regular and bigger than $|W|$ and $|V^+_{\xi_0}|$.
 
 Consider the set $V^*=V^-\cup W\cup V^+_{<\xi_0}$. We can apply Lemma \ref{lm:degreecontroll} in $D\uhr V^*$ for $W\setm V^+_{<\xi_0}$ and $\lambda$. So, we can find a set $W^*$ of size $<\lambda$ with $W\setm V^+_{<\xi_0}\subseteq W^*\subseteq V^*$ so that all arcs point into $W^*$ from $V^*\setm W^*$ after reversing $<\lambda$ many cycles $\mc C^*$. 
 
 We claim that  now all arcs point from $V^{--}=V^*\setm W^*$ to $V^{++}=V\setm (V^*\setm W^*)=(V^+\setm V^+_{<\xi_0})\cup W^*$. Indeed, suppose $u\in   V^{--}$ and $v\in V^{++}$. If $v\in W^*$ then $uv$ is an arc by the construction of $W^*$. 
 
 On the other hand, suppose $v\in V^+\setm (W^*\cup V^+_{<\xi_0})$ so in particular, $v\notin W$ and $v\notin V^+_{<\xi_0}$. Note that $V^{--}\subseteq (V^-\setm W)\cup V^+_{<\xi_0}$. If $u\in V^-\setm W$ then $u,v\notin W$ and so $uv$ is an arc since $W$ pinned down all arcs from $V^+$ to $V^-$ and arcs meeting $v$ were not used in $\mc C^*$ (Since each cycle in $\mc C^*$ is contained in $D\uhr V^*$). Finally, if $u\in V^+_{<\xi_0}$ then $uv$ is an arc by the definition of the $\la V^+_\xi\ra_{\xi<\mu}$ decomposition (this again was not changed when reversing the cycles of $\mc C^*$).
 
 Hence, we showed that all arcs point from $V^{--}$ to $V^{++}$ and so strong components are either in $V^{--}$ or $V^{++}$. Note that $|V^-\Delta V^{--}|<\kappa$ and $|V^+\Delta V^{++}|<\kappa$ and reversing $\mc C^*$ only changed $<\kappa$ many arcs. So out-degrees are $<\kappa$ in $V^{--}$ since out-degrees were $<\kappa$ in $V^-$.   Similarly, in-degrees are $<\kappa$ in $V^{++}$ since in-degrees were $<\kappa$ in $V^+$. 
 
\end{proof}

\begin{cor}\label{cor:third}
  If $D\in \mf {UT}_\kappa$ then there is $\mc C\in \rs D$ so that strong components of $\rev{D}{\mc C}$ have size $<\kappa$.
\end{cor}
\begin{proof}
 Apply Theorem \ref{thm:smallcomp} first and then Corollary \ref{cor:second}.
\end{proof}

\begin{cor}\label{cor:final} 
  If $D\in \mf T_\kappa$ then there is $\mc C\in \rs D$ so that strong components of $\rev{D}{\mc C}$ have size $<\kappa$.
\end{cor}
\begin{proof}
 Apply Corollary \ref{cor:first} first and then Corollary \ref{cor:third}.
\end{proof}

Now, the Main Theorem follows from Corollary \ref{cor:final} by induction on $\kappa$.

%

 \section{Appendix A: finite tournaments and the {\DD dichromatic} number}\label{sec:finite}
 
 As we mentioned,  Thomass\'e and Charbit already had different arguments to show that  $\dchr{\rev{D}{\mc C}}\leq 2$ for some $\mc C\in \rs D$ whenever $D$ is a finite digraph. {\DD Furthermore, in \cite{gyarfas_rev}, the authors proved that if $D$ is a finite tournament on vertices $v_1,v_2\dots v_n$ then one can find $\mc C\in \rs D$ so that both $\{v_i:1\leq i\leq n \textmd{ is even}\}$ and $\{v_i:1\leq i\leq n \textmd{ is odd}\}$ are acyclic in $\rev{D}{\mc C}$.}
 
 Now, we present yet another variation.

 \begin{theorem}\label{thm:finite}
  Suppose that $D$ is a finite tournament and $W_0$ is an acylic set of vertices. Then there is a $W\supseteq W_0$ and $\mc C\in \rs D$ so that both $W$ and $V \setm W$ are acyclic in $\rev{D}{\mc C}$. 
  
 \end{theorem}
\begin{proof}
 Fix $D=(V,A)$ and consider $$\mc W=\{W\subseteq V:W_0\subseteq W \textmd{ and } (\rev{D}{\mc C})\uhr W \textmd{ is acyclic for some }\mc C\in \rs D\}.$$ Select a maximal $W\in \mc W$ and fix $\mc C_0\in \rs D$ so that $(\rev{D}{\mc C_0})\uhr W$ is acyclic. In turn, any further cycle reversion which keeps $W$ acyclic also preserves its maximality. Our goal will be to make the complement of $W$ acyclic by further cycle reversions.
 
 Now, let $\tri_W$ be a linear order on $W$ so that all arcs $uv$ in $(\rev{D}{\mc C_0})\uhr W$ satisfy $u\tri_W v$. Also, find a linear order $\tri$ on $V\setm W$ and $\mc C_1\in \rs{\rev{D}{\mc C_0}}$ so that 
 \begin{enumerate}
 \item $(\rev{D}{\mc C})\uhr W$ is still acyclic,
  \item  the size of the set of $\tri$-backward arcs $$F_{\tri,\mc C}=\{yx\in A((\rev{D}{\mc C})\uhr V\setm W):x\tri y\}$$ is minimal, and 
  \item  the size $k=k_{\tri,\mc C}$ of the smallest sequence of $\tri$-successors $x=x_0\tri x_1\tri \dots \tri x_{k}=y$ for some arc $yx\in F_{\tri,\mc C}$ is also minimal
 \end{enumerate}
 where $\mc C=\mc C_0\smf \mc C_1$. 
 
 We would like to show that $F=F_{\tri,\mc C}=\emptyset$ which in turn yields that $(\rev{D}{\mc C})\uhr V\setm W$ is also acyclic. Suppose that $F\neq \emptyset$ and let $yx\in F$ so that the sequence of $\tri$-successors $x=x_0\tri x_1\tri \dots \tri x_{k}=y$ between $x$ and $y$ is minimal. 
 
 First, if $k=1$ i.e. $y$ is the $\tri$-successor of $x$ then we can modify $\tri$ into $\tri^*$ by switching the order of $x$ and $y$. That is, $\tri^*=\tri\setm \{(x,y)\}\cup \{(y,x)\}$ and it is easy to see that $F_{\tri^*,\mc C}=F_{\tri,\mc C}\setm \{yx\}$. This contradicts the minimality of $F$.
 
 Second, suppose that $k>1$; in turn, $xx_{k-1}$ and $x_{k-1}y$ are arcs in $\rev{D}{\mc C}$. Since $W$ is maximal acyclic in $\rev{D}{\mc C}$ as well, we can find $\tri_W$-successors $u\tri_W v\in W$ so that $yuv$ is a directed 3-cycle.  
 
 \begin{claim}
   $x_{k-1}u\in A(\rev{D}{\mc C})$.
 \end{claim}
\begin{proof}Indeed, otherwise, we can
 \begin{enumerate}[(i)]
  \item   define $\mc C^*$ by adding the directed 3-cycle $yux_{k-1}$ to $\mc C$,
  \item define  $\tri^*=\tri\setm \{(x_{k-1},y)\}\cup \{(y,x_{k-1})\}$.
\end{enumerate}
Now, it is easy to see that $F_{\tri^*,\mc C^*}=F_{\tri,\mc C}$ but $k_{\tri^*,\mc C^*}=k-1$ which contradicts the minimality of $k$.
 \end{proof}
 
 \begin{figure}[H]

 \psscalebox{0.6} 
{
\begin{pspicture}(0,-2.5488892)(15.36,2.5488892)
\psline[linecolor=black, linewidth=0.04, linestyle=dashed, dash=0.17638889cm 0.10583334cm, arrowsize=0.05291667cm 2.0,arrowlength=1.4,arrowinset=0.0]{->}(2.6,-1.2411108)(11.6,-1.2411108)
\psline[linecolor=black, linewidth=0.04, linestyle=dashed, dash=0.17638889cm 0.10583334cm, arrowsize=0.05291667cm 2.0,arrowlength=1.4,arrowinset=0.0]{->}(8.98,1.9788891)(11.6,1.9588891)
\rput[bl](0.6,1.9588891){\huge{$W$}}
\rput[bl](0.0,-1.4411108){\huge{$V\setminus W$}}
\psdots[linecolor=black, dotsize=0.2](7.4,1.9588891)
\psdots[linecolor=black, dotsize=0.2](8.8,1.9588891)
\psline[linecolor=black, linewidth=0.04, arrowsize=0.05291667cm 2.0,arrowlength=1.4,arrowinset=0.0]{->}(7.4,1.9588891)(8.6,1.9588891)
\rput[bl](7.2,2.358889){\huge{$u$}}
\rput[bl](8.8,2.358889){\huge{$v$}}
\psdots[linecolor=black, dotsize=0.2](4.2,-1.2411108)
\psdots[linecolor=black, dotsize=0.2](7.6,-1.2411108)
\rput[bl](12.4,1.9588891){\huge{$\tri_W$}}
\rput[bl](12.4,-1.0411109){\huge{$\tri$}}
\psline[linecolor=black, linewidth=0.04, arrowsize=0.05291667cm 2.0,arrowlength=1.4,arrowinset=0.0]{->}(7.6,-1.2411108)(8.8,-1.2411108)
\psdots[linecolor=black, dotsize=0.2](9.0,-1.2411108)
\rput[bl](9.2,-2.2){\huge{$y=x_k$}}
\rput[bl](6.6,-1.9){\huge{$x_{k-1}$}}
\rput[bl](2.8,-2.2){\huge{$x=x_0$}}
\psline[linecolor=black, linewidth=0.04, arrowsize=0.05291667cm 2.0,arrowlength=1.4,arrowinset=0.0]{->}(8.8,1.9588891)(9.0,-1.0411109)
\psline[linecolor=black, linewidth=0.04, arrowsize=0.05291667cm 2.0,arrowlength=1.4,arrowinset=0.0]{->}(9.0,-1.2411108)(7.5,1.7988892)
\rput{-128.21068}(9.591154,7.433888){\psarc[linecolor=black, linewidth=0.04, dimen=outer, arrowsize=0.05291667cm 2.0,arrowlength=1.4,arrowinset=0.0]{<-}(6.6,1.3888892){3.61}{0.0}{79.723976}}
\psline[linecolor=black, linewidth=0.04, linestyle=dashed, dash=0.17638889cm 0.10583334cm, arrowsize=0.05291667cm 2.0,arrowlength=1.4,arrowinset=0.0]{->}(7.6,-1.2011108)(7.4,1.7788892)
\psline[linecolor=black, linewidth=0.04, linestyle=dashed, dash=0.17638889cm 0.10583334cm](2.6,1.9588891)(7.4,1.9588891)
\rput{-308.16718}(-0.1843013,-5.8429446){\psarc[linecolor=black, linewidth=0.04, dimen=outer, arrowsize=0.05291667cm 2.0,arrowlength=1.4,arrowinset=0.0]{<-}(5.92,-3.111111){2.55}{0.0}{79.723976}}
\end{pspicture}
}
 \caption{How to make $V\setm W$ acyclic}
\label{fig:finiteproof}
 
 \end{figure}
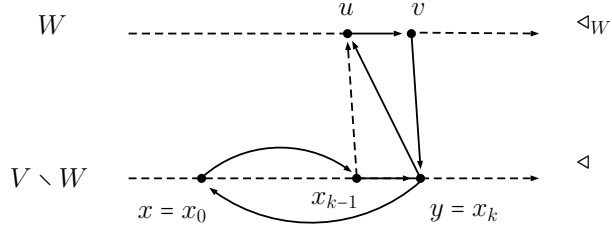
 
Finally, consider the directed 5-cycle $C^*=yxx_{k-1}uv$ and let   $\mc C^*=\mc C\smf \la C^*\ra$. On one hand, $W$ remains acyclic in $\rev{D}{\mc C^*}$ since $u,v$ were $\tri_W$-successors. Furthermore,  $$F_{\tri,\mc C^*}=\left(F_{\tri,\mc C}\setm \{yx\}\right) \cup \{x_{k-1}x\}$$ however $k_{\tri,\mc C^*}=k-1$ which again contradicts the minimality of $k$. 

This contradiction proves that $F=\emptyset$ to start with, and so $\rev{D}{\mc C}$ is covered by the two acyclic sets $W$ and $V\setm W$.  
 
 %
%

\end{proof}

\subsection*{On the number of cycles}

{\DD Now, let us recall Charbit's elegant argument for proving Theorem \ref{thm:charbit} \cite{charbit}; we do this in part so that we can reflect on the number of cycles required in such a reversal sequence for arbitrary digraphs too.}

Given a finite digraph $D$, let us arrange the vertices of $D$ evenly along a circle of perimeter 1. Now, each arc $a$ of $D$ has a length $\ell(a)$ i.e. the length of the arc on the circle connecting two vertices, and each arc either points forward or backward. Let $\sigma(D)=\sum_{a\in A}\ell(a)$ and note that $\frac{m}{n}\leq \sigma(D)\leq m$.

Now, we will say that a cycle $C\in \mb D$ is good if $C$ has at least as many forward arcs as backward arcs. First, if $C\in \mb D$ is not good than reversing $C$ lowers $\sigma(D)$ by at least $\frac{1}{n}$, that is,  $\sigma(\rev{D}{C})\leq \sigma(D)-\frac{1}{n}$. Since this reduction is only possible $(m-\frac{m}{n})/\frac{1}{n}=(n-1)m$ times, there is some $\mc C\in \rs D$ of size at most $(n-1)m$ so that all cycles of $\rev{D}{\mc C}$ are good.

The final piece of the proof is the following theorem.

\begin{theorem}\label{charbitthm}\cite[Theorem 4.4]{charbit}
If each cycle of $D$ is good then $\dchr D\leq 2$.
\end{theorem}

We make the next definition: let
$$
 \crs{D}=\min\{|\mc C|:\mc C\in \rs D\textmd{ and }\dchr{\rev{D}{\mc C}}\leq 2\}.
 $$

\begin{cor}\label{cor:bound}
 $\crs{D}\leq (n-1)m$ for any digraph $D$ with $n$ vertices and $m$ arcs.
\end{cor}

{\DD For finite tournaments, our Theorem \ref{thm:finite} gives a similar bound, but the situation was analysed in \cite{gyarfas_rev} and one can say more.

\begin{theorem}[\cite{gyarfas_rev}]$\crs{D}\leq \frac{(n-1)(n-2)}{2}$ for any tournament $D$ on $n$ vertices (and this is witnessed by a sequence of 3-cycles).
\end{theorem}

Furthermore, there is a lower bound for the number of 3-cycles required in such a reversal:  if $D$ is the Paley tournament on $n$ vertices then we need the reversal of at least $(1-o(1)){n^2}/{32}$ many 3-cycles to lower the dichromatic number to 2 \cite[Theorem 8]{gyarfas_rev}. We are not aware of a lower bound for $\crs{D}$ though. }

\medskip


Let us present now a simple iterative process to produce a digraph with large {\DD dichromatic} number which could serve as another test case. Given a digraph $D=D^c$ with {\DD dichromatic} number $c$ and digirth $\ell$ one can construct a digraph $D^{c+1}$ with chromatic number $\geq c+1$ and digirth $\ell$. Indeed, take disjoint copies $\{D_i:i< \ell\}$ of $D$ and let $D^{c+1}$ be the union of $\{D_i:i<\ell\}$  with the additional edges $uv$ where $u\in V(D_i),v\in V(D_{i+1})$ for $i<\ell-1$ and $u\in V(D_{\ell-1}),v\in V(D_{0})$. If there is an acyclic colouring of $D^{c+1}$ with $c$ colours then each colour must appear in each copy of $D^c$; however, this means that there are monochromatic cycles of length $\ell+1$ with meet each $D_i$ in exactly one vertex. It would be interesting to see the values of $\crs{D^c}$ calculated when starting from a 3-cycle with $\ell=3$.

\medskip

%

\subsection*{Complexity considerations}

Regarding complexity, whether Charbit's proof can be carried out in polynomial time comes down to the question if we can find, in polynomial time, a cycle in $D$ which is not good. Define the weight function $w:A\to \{-1,1\}$ as $w(a)=1$ if $a$ is forward and  $w(a)=-1$ if $a$ is backward. Now, $C$ is good iff $\sum_{a\in A(C)}w(A)\geq 0$ so we look for an algorithm which finds negative cycles in polynomial time: this is done by the Bellman-Ford algorithm in $O(nm)$ time \cite{bellmann}. So we can find a reversal sequence $\mc C\in \rs D$ in polynomial time such that $\dchr{\rev{D}{\mc C}}\leq 2$. 

{\DD Regarding tournaments, one can find a sequence of 3-cycles $\mc C\in \rs D$ in $O(n^2)$ time such that $\dchr{\rev{D}{\mc C}}\leq 2$ \cite[Theorem 9]{gyarfas_rev}. }

\medskip

At this point, we don't know how hard it is to calculate $\crs D$ i.e. if one can decide if $\crs D \leq k$ in polynomial time for a fixed $k$ (even for tournaments).


 \section{Appendix B: the structure of reversal sequences} \label{app:structure}
 
      Next, we would like to better understand an infinite reversal sequence. If $\mc C\in \rs D$ and $e\in E(D)$ then let $$\St(\mc C, e)=\{C\in \ran(\mc C):e\in E(C)\}.$$ Now, let $\mc C^{(e)}$ be the minimal $\mc E\subseteq \ran(\mc C)$ so that  
     \begin{enumerate}
      \item      $\St(\mc C, e)\subseteq \mc E$, and 
      \item if $f\in E(C)$ for some $C\in \mc E$ then  $\St(\mc C, f)\subseteq \mc E$ as well.
     \end{enumerate}

     \begin{obs}
      $\mc C^{(e)}$ is countable and $E(\mc C^{(e)})\cap E(\mc C^{(f)})=\emptyset$ or $\mc C^{(e)}=\mc C^{(f)}$ for all $e,f\in E$.
     \end{obs}
    {\DD \begin{proof}
      Indeed, $|\mc C^{(e)}|\leq \aleph_0$ follows from $\mc C$ being locally countable. 
      
      The second statement is an immediate corollary of the following observation: $g\in E(\mc C^{(e)})$ iff there is a finite sequence of cycles $C'_0\dots C'_n\in \ran(\mc C)$ so that $e\in C'_0$, $g\in C'_n$ and $E(C'_i)\cap E(C'_{i+1})\neq \emptyset$ for all $i<n$.
      
     \end{proof}}

At this point, $\mc C^{(e)}$ is only a set but it inherits a well order from  $\mc C$, in some countable order type.

     \begin{obs}
      $\mc C^{(e)}\in \rs D$ for all $e\in E$.
\end{obs}{\DD
\begin{proof}
 The proof is an easy induction on the length of $\mc C^{(e)}$.
\end{proof}

Actually, we can say even more: if $\mc C^{(e)}\neq \mc C^{(f)}$ then both $\mc C^{(e)}\smf \mc C^{(f)}$ and $\mc C^{(f)}\smf \mc C^{(e)}$ are reversal sequences in $D$ and $\rev{D}{(\mc C^{(e)}\smf \mc C^{(f)})}=\rev{D}{(\mc C^{(f)}\smf \mc C^{(e)})}$. 

Now, the next corollary easily follows.

\begin{cor}\label{l:canonical}
 If $\mc C\in \rs  D$ and $\mc C^*$ is any enumeration of $\{\mc C^{(e)}:e\in E(D)\}$ then $\mc C^*\in \rs D$ and  $\rev{D}{\mc C}=\rev{D}{\mc C^*}$.
\end{cor}}

%
%
%
%


Ultimately, Corollary \ref{l:canonical} tells us that the effect of reversing $\mc C$ can be reproduced by reversing countable sets of cycles \emph{independently and in any order we choose.} For these reasons, we will call $\la \mc C^{(e)}\ra_{e\in E}$ the \emph{canonical decomposition} of $\mc C$.

%
%

%

\medskip

Corollary \ref{l:canonical} also implies that  if $|\mc C|=\kappa>\oo$ then there is a rearrangement $\mc C^*$ of $\mc C$ of \emph{type $\kappa$} so that  $\rev{D}{\mc C}=\rev{D}{\mc C^*}$. How about countable sequences? The next theorem further simplifies the picture: the countable reversal sequences can be rearranged into type $\oo$ cycle reversions.

\begin{theorem}\label{thm:tpoo}
 Suppose that $\mc C\in \rs D$ is countable. Then there is a rearrangement $\mc C^*\in \rs D$ of $\mc C$ of type $\leq\oo$ so that $\rev{D}{\mc C}=\rev{D}{\mc C^*}$.
\end{theorem}

By a rearrangement we simply mean that the sequences $\mc C$ and $\mc C^*$  contain the same cycles (with the same multiplicity).


\begin{proof}
 We will prove by induction on the length $\xi$ of $\mc C$. If $\xi\leq \oo$ then we can take $\mc C^*=\mc C$.
 
 Suppose that we proved the statement for $\xi$ and we would like to step up to $\xi+1$. Let $\xi_0$ be the largest limit ordinal $\leq \xi$. Now, there is a type $\leq \omega$ sequence $\mc C^*$ which is equivalent to $\mc C\uhr \xi_0$. There is an $\ell<\oo$ so that no edge from the cycles $\mc C\uhr (\xi+1\setm \xi_0)$ appears in $\mc C^*\uhr \oo\setm \ell$; this is because $\mc C\uhr (\xi+1\setm \xi_0)$ is finite. So we define $\mc C^{**}=(\mc C^*\uhr \ell)\smf (\mc C\uhr (\xi+1\setm \xi_0))\smf (\mc C^*\uhr (\oo\setm \ell))$. It is easy to see that $\mc C^{**}\in \rs D$, $\mc C^{**}$ has type $\oo$ and $\rev{D}{\mc C}=\rev{D}{\mc C^{**}}$.
 
 Now the limit stages: let $\xi_n$ denote a type $\oo$ cofinal sequence in $\xi$. Let $\mc C_0\in \rs D$ be of type $\leq \oo$ equivalent to $\mc C\uhr \xi_0$, and inductively find $\mc C_{n+1}\in \rs{\rev{D}{\mc C_{\leq n}}}$ of type $\leq \oo$ equivalent to $\mc C\uhr \xi_{n+1}\setm \xi_n$. So $\la \mc C_n\ra_{n\in \oo}$ has type $\leq \oo\cdot \oo$ and is equivalent to $\mc C$, i.e., $\rev{D}{\la \mc C_n\ra_{n\in \oo}}=\rev{D}{\mc C}$. To make notation more simple, we assume that $\mc C=\la \mc C_n\ra_{n\in \oo}$.
 
 Start by listing the countably many edges that appear in the cycles of $\mc C$ in type $\omega$ as $\{a_k:k\in \oo\}$. Our first goal is to find a finite $\mc C^*_0\in \rs D$, and some $\mc D_0\in \rs{\rev{D}{\mc C^*_0}}$ of type $\leq \oo\cdot \oo$ so that 
 \begin{enumerate}
 \item both  $\mc C^*_0,\mc D_0$ are subsequences of $\mc C$,
  \item $\mc \mc C^*_0\smf \mc D_0$ is equivalent to $\mc C$, and 
  \item $a_0$ does not appear in cycles from $\mc D_0$.
 \end{enumerate}
Then we look at $a_1$ in $\rev{D}{\mc C^*_0}$ and repeat the process to find a finite subsequence $\mc C^*_1$ of $\mc D_0$ and $\mc D_1$ subsequence of $\mc D_0$ so that $\mc C^*_1\smf \mc D_1$ is equivalent to $\mc D_0$ but $a_1$ (or $a_0$) does not appear in $\mc D_1$ any more. In the end, the sequence $\mc C^*=\la \mc C^*_k:k<\oo\ra$ has type $\leq \oo$ and is equivalent to $\mc C$.
 
 We present the construction of $\mc C^*_0$ and $\mc D_0$ in detail; the rest is strictly analoguous.
 
 Let $n_0$ be the maximal index so that $a_0$ appears in some cycle from $\mc C_n$. Let $\mc F_{n_0}$ be the finite initial segment of $\mc C_{n_0}$ so that $a_0$ does not appear in $\mc C_{n_0}\setm \mc F_{n_0}$. Let $\mc F_{n_0-1}$ be the finite initial segment of $\mc C_{n_0-1}$ so that edges from $\mc F_{n_0}$ do not appear in cycles from 
 $\mc C_{n_0-1}\setm \mc F_{n_0-1}$. Given $\mc F_{n_0},\mc F_{n_0-1}\dots \mc F_{n_0-k+1}$, we let $\mc F_{n_0-k}$ be the smallest finite initial segment of $\mc C_{n_0-k}$ so that edges from $\mc F_{n_0},\mc F_{n_0-1}\dots \mc F_{n_0-k+1}$ do not appear in $\mc C_{n_0-k}\setm \mc F_{n_0-k}$. This defines $\mc F_0, \mc F_1\dots \mc F_{n_0}$ and we let $$\mc C_0^*=\mc F_0\smf \mc F_1\smf \dots \smf \mc F_{n_0}$$
 and $$\mc D_0=(\mc C_0\setm \mc F_0)\smf  (\mc C_1\setm \mc F_1)\smf \dots \smf (\mc C_{n_0}\setm \mc F_{n_0})\smf \mc C_{n_0+1}\smf\mc C_{n_0+2}\smf \dots$$

 First, a simple induction on $k=0,1\dots n_0$ proves that $\mc F_0\smf \mc F_1\smf \dots \smf \mc F_{k}\in \rs D$ and so $\mc C_0^*\in \rs D$ as well. Similarly, $(\mc C_0\setm \mc F_0)\smf  (\mc C_1\setm \mc F_1)\smf \dots \smf (\mc C_{k}\setm \mc F_{k})\in \rs{\rev{D}{\mc C^*_0}}$ is proved by induction $k=0,1\dots n_0$ and in turn $\mc D_0\in \rs{\rev{D}{\mc C^*_0}}$ follows.
 
 The only thing left to show is that $\mc \mc C^*_0\smf \mc D_0$ is equivalent to $\mc C$. However this is trivial: $\mc C^*$ is a rearrangement of $\mc C$ and the direction of an arc after reversing by either sequence is simply decided by the number cycles the arc appears in.
\end{proof}

{\DD Finally, let us mention the following corollary:}

 \begin{cor}
  If $\mc C\in \rs D$ is arbitrary and $F\subseteq A(\rev{D}{\mc C})$ is finite then there is a finite $\mc C_F\in \rs D$ so that $F\subseteq A(\rev{D}{\mc C_F})$ already.
 \end{cor}

\section{Appendix C: reversing triangles versus reversing arbitrary cycles}\label{app:tri}

{\DD It is an easy exercise (already noted in \cite{gyarfas_rev}) that in any tournament $D$, the reversal of a cycle of length $k$ is equivalent to the reversion of $k-2$ many 3-cycles. In particular, the following holds.}

\begin{obs}
 Suppose that $D$ is a tournament and $\mc C\in \rs D$ is finite. Then there is sequence of 3-cycles $\mc C_\Delta\in \rs D$ so that $\rev{D}{\mc C}=\rev{D}{\mc C_\Delta}$.
\end{obs}

{\DD Now, what can we say about infinite sequences? In particular, can we make sure that when we substitute a cycle by a sequence of 3-cycles the new sequence remains locally finite? We claim that the answer is yes, at least for countable sequences.

First, we need a slightly technical but easy statement.}



\begin{prop}\label{prop:only3}
 Suppose that $D$ is a tournament and $C\in \mb C(D)$. Then for any vertex $v\in V(C)$ there is a sequence of 3-cycles $\mc C_v$ so that $\rev{D}{C}=\rev{D}{\mc C_v}$ and any edge $e\in E(\mc C_v)$ either contains $v$ or $e\in E(C)$. 
\end{prop}

\begin{proof}
 We prove by induction on $k=|C|$ for all tournaments $D$ simultaneously. If $k=3$ then we can trivially take $\mc C=\la C\ra$.
 
 Suppose that $k>3$ now and we distinguish two cases: first, suppose that $v_0v_2\in A(D)$ where $C$ is on vertices $v_0 v_1\dots v_{k-1}$. Now, $C_0=v_0v_2v_3 \dots v_{k-1}$ is a $k-1$-cycle so there is a sequence $\mc C_0\in\rs D$ of 3-cycles (by the inductive hypothesis) so that  $\rev{D}{C_0}=\rev{D}{\mc C_0}$. Let $C_1=v_0v_1v_2$ and note that $C_1\in \mb C(\rev{D}{C_0})$ and $$\rev{D}{C}=\rev{D}{\la C_0,C_1\ra}=\rev{D}{(\mc C_0\smf \la C_1\ra)}.$$ In turn, $\mc C=\mc C_0\smf \la C_1\ra$ is the desired sequence of 3-cycles.
 
 Second, suppose that $v_2v_0\in A(D)$. Now, we reverse $C_0=v_0v_1v_2$ first and note that $C_1=v_0v_2\dots v_{k-1}$ is a $k-1$-cycle in $\rev{D}{C_0}$. In turn, the induction applies and we can find a sequence of 3-cycles $\mc C_1\in \rs{\rev{D}{C_0}}$ so that $$\rev{D}{C}=\rev{D}{\la C_0,C_1\ra}=\rev{D}{(\la C_0\ra\smf \mc C_1)}.$$ So $\mc C=\la C_0\ra\smf \mc C_1$ is the desired sequence of 3-cycles. 
 \end{proof}



 \begin{cor}  If $D$ is a tournament and $\mc C\in \rs D$ is countable then there is a type $\leq \oo$ sequence of 3-cycles $\mc C_\Delta\in \rs D$ so that  $\rev{D}{\mc C}=\rev{D}{\mc C_\Delta}$. 
 \end{cor}
 We already proved this for finite $\mc C$; the slight difficulty now comes from arranging that the sequence of triangles provided by Proposition \ref{prop:only3} remains locally finite. 
 
 \begin{proof}
  First, we can suppose that $\mc C$ has type $\omega$ by Theorem \ref{thm:tpoo} i.e. $\mc C=\la C_n\ra_{n\in \oo}$. 
  
  \begin{claim}\label{clm:finunion}
   There are finite sets of vertices $W_0\subseteq W_1\subseteq \dots$ and $n_0<n_1<\dots $ in $\mb N$ so that 
  \begin{enumerate}
   \item $\bigcup V(\mc C)=\bigcup\{W_i:i<\oo\}$, and 
   \item $A(C)\cap W_i^2=\emptyset$ for any $C\in \ran(\mc C\uhr \oo\setm n_i)$ and $i<\oo$.
  \end{enumerate}
  \end{claim}
This is easily done using the local finiteness and countable size of $\mc C$.

Now, note that (2) implies that if $n_i\leq n< n_{i+1}$ then we can fix a vertex $v_n\in V(C_n)\setm W_{i}$. Next, let $\mc D_n$ denote the sequence of 3-cycles on vertex $v_n$ given by Proposition \ref{prop:only3} equivalent to $C_n$ in $\rev{D}{\mc C\uhr n}$. Keep in mind that any edge $e$ of a 3-cycle from $\mc D_n$ is either an edge of $C_n$ or contains $v_n$.

We claim that $\mc C_\Delta=\la \mc D_n \ra_{n\in \oo}\in \rs D$ and $\rev{D}{\mc C}=\rev{D}{\mc C_\Delta}$. The fact that $\la \mc D_n\ra_{n<m}\in \rs D$ and $\rev{D}{\mc C\uhr m}=\rev{D}{\la \mc D_n\ra_{n<m}}$ easily follows from the choice of $\mc D_n$.

So, we need that $\mc C_\Delta$ is locally finite. If $e\in E(\mc C_\Delta)$ then there is an $i$ so that $e\in [W_i]^2$. If $n\geq n_i$ then $e\notin E(C_n)$ and $v_n\notin e$. In particular, no triangle from $\mc D_n$ contains $e$. This proves that $e$ can only appear in $\la D_n \ra_{n< n_i}$, and so $\mc C_\Delta$ is locally finite.

 \end{proof}
 
 {\DD We do not know if the above corollary extends to uncountable sequences. However,} note that simply applying Proposition \ref{prop:only3} will not suffice in the uncountable case: let $D$ be a tournament of vertices $\omg$ and suppose that {\DD $C_{\alpha,n}=(\alpha, 2n, \alpha+1,  2n+1)$} is a 4-cycle for all limit $\alpha\in \omg$ and $n<\oo$. The sequence $\mc C=\la C_{\alpha,n}\ra_{\alpha\in \lim(\omg),n\in \oo}$ consists of edge-disjoint cycles so $\mc C\in \rs D$. 
 
 If we triangulate using Proposition \ref{prop:only3}, then we select a diagonal of the cycles with each triangle. However, if  $e_{\alpha,n}$ is a diagonal of $C_{\alpha,n}$ then $\{e_{\alpha,n}:\alpha\in \lim(\omg),n\in \oo\}$ is not locally finite. Indeed, if there are uncountably many $\alpha$ so that $e_{\alpha,n}\in [\oo]^2$ for some $n\in \oo$ then there is a single $e\in [\oo]^2$ so that $e=e_{\alpha,n}$ for  uncountably many $\alpha$. So, for almost all $\alpha$, we selected $e_{\alpha,n}=\{\alpha,\alpha+1\}$ for all $n\in \oo$.

\section{Open problems}

 In our (biased) opinion, the two main questions which remained open are whether the Main Theorem extends to all digraphs and if Thomass\'e's conjecture holds for arbitrary digraphs.
 
  \begin{prob}
  Suppose that $D$ is an arbitrary digraph. Is there a $\mc C\in \rs D$ so that $\rev{D}{\mc C}$ has finite strong components?
 \end{prob}
 
  \begin{prob}
  Suppose that $D$ is an arbitrary digraph. Is there a $\mc C\in \rs D$ so that $\rev{D}{\mc C}$ has {\DD dichromatic} number $\leq 2$?
 \end{prob}
 
 Both questions are open even for countably infinite digraphs.
 
 \medskip
 
 Now, regarding finite digraphs, it would be very interesting to learn more about the reversal sequences that are used to lower the {\DD dichromatic} number.
 
   \begin{prob}
  Suppose that $D$ is a finite digraph. Find a lower bound for $\crs D$ i.e. the minimal length of a sequence $\mc C\in \rs D$ so that $\dchr{\rev{D}{\mc C}}\leq 2$.
 \end{prob}
 
   \begin{prob}
 Is the bound $\crs D\leq (n-1)m$ from Corollary \ref{cor:bound} sharp?
 \end{prob}

 
 {\DD Regarding the complexity of finding  $\crs D$ the obvious problem is NP-hardness.}
 
    \begin{prob}
Given a digraph (or tournament) $D$ and $k\in \mb N$, is the problem of deciding if $\crs D\leq k$  NP-hard?
\end{prob}

 \medskip
 
{\DD As we mentioned in Appendix C, we could not answer the next question for arbitrary uncountable reversal sequences.}
 
 \begin{prob}
  Suppose that $D$ is a tournament and $\mc C\in \rs D$ is arbitrary. Is there a sequence of 3-cycles $\mc C_\Delta\in \rs D$ so that $\rev{D}{\mc C}=\rev{D}{\mc C_\Delta}$?
 \end{prob}

 {\DD It is also reasonable to ask if the proof of the Main Theorem can be carried out using 3-cycles only.}
 
  \medskip
 
  {\DD Finally, we would like to mention two old and well known open problems from the theory of dichromatic number (independent of cycle reversions). The first is due to V. Neumann-Lara:}
 
 \begin{con} $\dchr{D}\leq 2$ for any planar digraph $D$.  
 \end{con}

 Quite recently, Z. Li and B. Mohar \cite{li} showed that the conjecture holds for digraphs of digirth at least 4.

 {\DD The second question is from P. Erd\H os and Neumann-Lara \cite{ENL, nessparse}:

\begin{con}\label{ENLconj} There is a function $f:\mb N\to \mb N$ so that $\chr G\geq f(k)$ implies that $\dchr{D}\geq k$ for some orientation $D$ of $G$.
\end{con}

Note that any graph $G$ with $\chr G\geq 3$ must contain a cycle and hence there is an orientation $D$ of $G$ with a directed cycle i.e. $\dchr D\geq 2$. In turn $f(2)=3$ but no other value of the function $f$ is currently known. Somewhat surprisingly, for graphs with chromatic number and size $\aleph_1$, one can say more about possible orientations and the dichromatic number \cite{dsoukup}.}

 \end{document}